\documentclass{birkjour}

\usepackage[utf8]{inputenc}
\usepackage[T1]{fontenc}
\usepackage[english]{babel}
\usepackage{mathtools,amsmath,amssymb,amsfonts,amsthm,mathrsfs}
\mathtoolsset{showonlyrefs}
\usepackage{enumitem}

\usepackage{geometry}
\usepackage{graphicx,color}
\usepackage{tikz,pgfplots,pgf}
\usepackage[caption=false]{epsfig}
\usepackage{subcaption}
\usepackage[font=small]{caption}
\usepackage[pdftex]{hyperref}

\newtheorem{theorem}{Theorem}[section]
\newtheorem{lemma}[theorem]{Lemma}
\newtheorem{proposition}[theorem]{Proposition}
\newtheorem{corollary}[theorem]{Corollary}
\theoremstyle{definition}
\newtheorem{remark}[theorem]{Remark}

\newtheorem{example}[theorem]{Example}

\numberwithin{equation}{section}

\begin{document}

\title[Bifurcation of closed orbits of Hamiltonian systems]{Bifurcation of closed orbits of Hamiltonian systems \\with application to geodesics of the Schwarzschild metric}

\author[A.~Boscaggin]{Alberto Boscaggin}

\address{
Department of Mathematics ``Giuseppe Peano'', University of Torino\\
Via Carlo Alberto 10, 10123 Torino, Italy}

\email{alberto.boscaggin@unito.it}

\author[W.~Dambrosio]{Walter Dambrosio}

\address{
Department of Mathematics ``Giuseppe Peano'', University of Torino\\
Via Carlo Alberto 10, 10123 Torino, Italy}

\email{walter.dambrosio@unito.it}

\author[G.~Feltrin]{Guglielmo Feltrin}

\address{
Department of Mathematics, Computer Science and Physics, University of Udine\\
Via delle Scienze 206, 33100 Udine, Italy}

\email{guglielmo.feltrin@uniud.it}

\thanks{Work written under the auspices of the Grup\-po Na\-zio\-na\-le per l'Anali\-si Ma\-te\-ma\-ti\-ca, la Pro\-ba\-bi\-li\-t\`{a} e le lo\-ro Appli\-ca\-zio\-ni (GNAMPA) of the Isti\-tu\-to Na\-zio\-na\-le di Al\-ta Ma\-te\-ma\-ti\-ca (INdAM). The first and third authors are supported by the INdAM-GNAMPA Project 2023 ``Analisi qualitativa di problemi differenziali nonlineari'' and PRIN 2022 ``Pattern formation in nonlinear phenomena''.
\\
\textbf{Preprint -- October 2023}} 

\subjclass{34C25, 70H08, 70H12, 83A05, 83C10.}

\keywords{Nearly integrable Hamiltonian systems, closed orbits, apsidal angle, time-map, central force problems, Schwarzschild metric, Poincar\'e--Birkhoff theorem.}

\date{}

\dedicatory{}

\begin{abstract}
We investigate bifurcation of closed orbits with a fixed energy level for a class of nearly integrable Hamiltonian systems with two degrees of freedom. More precisely, we make a joint use of Moser invariant curve theorem and Poincar\'e--Birkhoff fixed point theorem to prove that a periodic non-degenerate invariant torus $\mathcal{T}$ of the unperturbed problem gives rise to infinitely many closed orbits, bifurcating from a family of tori accumulating onto $\mathcal{T}$. The required non-degeneracy condition, which is nothing but a reformulation of the usual non-degeneracy condition in the isoenergetic KAM theory, is expressed in terms of the derivative of the apsidal angle with respect to the angular momentum: in this way, tools from the theory of time-maps of nonlinear oscillators can be used to verify it in concrete problems.
Applications are given to perturbations of central force problems in the plane, and to equatorial geodesic dynamics for perturbations of the Schwarzschild metric.
\end{abstract}

\maketitle
	
\section{Introduction}\label{section-1}

In classical mechanics, the motion of a small point mass around a fixed massive body is modeled by the Kepler problem.
Due to its peculiar degeneracies, using tools of Hamiltonian perturbation theory to investigate perturbed Kepler dynamics is typically quite hard. However, as shown in \cite{BoDaFe-21,BDF4}, this superintegrability is broken if some relativistic approximations of the Kepler equation are considered.

In the context of general relativity, the analogue of the Kepler problem consists in the study of geodesic dynamics of the Schwarzschild metric.
As far as we known, the idea of using tools of Hamiltonian perturbation theory to investigate perturbations of the Schwarzschild metric was first developed in the interesting papers \cite{Xue-pp,Xue-21} (see also \cite{Mo-92} for a proof of chaos, using Melnikov method); however, 
a complete analytic proof of the required non-degeneracy is not given in \cite{Xue-21}. 

Motivated by these contributions, in this paper we take advantage of a time-map reformulation of the non-degeneracy condition to rigorously prove bifurcation of equatorial timelike closed geodesics for a stationary and equatorial perturbation of the Schwarzschild metric. 

As a further application of our general theory, we can also deal with bifurcation of closed orbits, with fixed energy, for perturbations of central force problems in the plane.

\subsection{State of the art}

According to Poincar\'e \cite{Po-1892}, investigating the behavior of Hamiltonian systems of the type
\begin{equation}\label{eq-hsintro}
\begin{cases}
\, \dot{\varphi}=\nabla \mathcal{K}_0 (I)+\varepsilon \, \nabla_I \widetilde{\mathcal{K}}(\varphi,I;\varepsilon),
\vspace{2pt}\\
\, \dot{I}=-\varepsilon \, \nabla_\varphi \widetilde{\mathcal{K}}(\varphi,I;\varepsilon), 
\end{cases}
\end{equation}
where $\varphi \in \mathbb{T}^N$, $I \in D \subset \mathbb{R}^N$ (the variable $\varphi$ is typically named angle, while the variable $I$ is the action) and $\varepsilon$ is a small real parameter, represents \emph{le probl\`eme g\'en\'eral de la dynamique}. A system of the above form is usually called nearly integrable, since the unperturbed problem can be explicitly solved: indeed, system \eqref{eq-hsintro} for $\varepsilon = 0$ reads as
\begin{equation}\label{eq-hsintro0}
\begin{cases}
\, \dot{\varphi}=\nabla \mathcal{K}_0 (I),
\vspace{2pt}\\
\, \dot{I}=0, 
\end{cases}
\end{equation}
so that each value $I^* \in D$ of the action gives rise to an invariant torus $\mathbb{T}^N \times \{I^*\}$ on which the dynamics of \eqref{eq-hsintro0} is completely described by the frequency vector $\omega^* = \nabla \mathcal{K}_0(I^*)$.

By far, the most remarkable property of the nearly integrable system \eqref{eq-hsintro} is the preservation, for $\varepsilon \neq 0$ and small, of many of such invariant tori, precisely the ones associated with a frequency vector $\omega^*$ satisfying a Diophantine condition: on these tori, the dynamics of \eqref{eq-hsintro} is quasi-periodic with highly non-resonant frequency. Results of this type, which are nowadays collected under the name of KAM theory (see for instance \cite{Du-14} for a friendly introduction to this wide topic), also require (besides quite severe conditions on the differentiability of the Hamiltonian vector field) a non-degeneracy condition for the unperturbed Hamiltonian $\mathcal{K}_0$, which typically reads as
\begin{equation}\label{eq-nondegintro}
\mathrm{det} \left( \nabla^2 \mathcal{K}_0(I^*) \right) \neq 0,
\end{equation} 
where the symbol $\nabla^2$ stands for the Hessian matrix, or as
\begin{equation}\label{eq-nondegintro2}
\mathrm{det} \,
\begin{pmatrix}
\nabla^2 \mathcal{K}_0(I^*) & \nabla \mathcal{K}_0(I^*)^{\intercal} 
\vspace{3pt} \\
\nabla \mathcal{K}_0(I^*) & 0
\end{pmatrix} 
\neq 0,
\end{equation}
if the perturbed tori are required to have a fixed energy level, independent from $\varepsilon$, see \cite[Appendix~8]{Ar-89} as well as \cite{BrHu-91} and the references therein.

On the opposite side of quasi-periodic tori with Diophantine frequency vector are tori carrying periodic solutions to system \eqref{eq-hsintro}: as well known, such tori generically disappear (i.e., they disintegrate) as soon as the parameter $\varepsilon$ is non-zero. In spite of this, traces of the unperturbed dynamics can still be found, for $\varepsilon$ small, in the form of surviving periodic orbits: results in this spirit, which typically provide a lower bound for the number of such periodic orbits in terms of a topological invariant of the torus, such as its Lusternik--Schnirelmann category, are usually called of bifurcation-type. They were obtained, in slightly different context and level of generality, in \cite{AmCoEk-87,BeKa-87,Bo-80,Chen-PhD,We-73,We-78}, all the corresponding proofs relying, ultimately, on some arguments of variational nature. More recently, in \cite{FoGaGi-16}, a result on this line was established using a higher dimensional version of the Poincar\'e--Birkhoff fixed point theorem previously proved, again via variational methods, in \cite{FoUr-17}. In all these bifurcation results, non-degeneracy conditions for the unperturbed Hamiltonian, again on the lines of \eqref{eq-nondegintro} or \eqref{eq-nondegintro2}, continue to play a crucial role.

In our recent paper \cite{BDF2}, we dealt with bifurcation of periodic solutions when the unperturbed system is a 
(isotropic) central force problem in the plane, namely
\begin{equation}\label{eq-cenintro}
\ddot x = V'(\vert x \vert) \frac{x}{\vert x \vert}, \quad x \in \mathbb{R}^2 \setminus \{0\}.
\end{equation}
As very well known, the above equation possesses two obvious first integrals, namely the energy $H = \tfrac{1}{2} \vert \dot x \vert^2 - V(\vert x \vert)$ and the angular momentum $L = x \wedge \dot x$, corresponding to the invariance of \eqref{eq-cenintro} for time-translations and plane-rotations, respectively. For a pair $(H,L)$ giving rise to non-rectilinear and non-circular periodic solutions, a symplectic change of variables $(x,\dot x) \mapsto (\varphi,I)$ can be constructed, transforming, locally, the Hamiltonian system naturally associated with \eqref{eq-cenintro} into a system of the form \eqref{eq-hsintro0}. For a concrete choice of the potential $V$, the non-degeneracy condition \eqref{eq-nondegintro}
is typically quite hard to be checked, since an explicit expression of the Hamiltonian $\mathcal{K}_0$ is not available; however, in \cite{BDF2} we proved that \eqref{eq-nondegintro} can be equivalently expressed in terms of two maps $T(H,L)$ and $\Theta(H,L)$ depending on the radial and angular dynamics of \eqref{eq-cenintro}. More precisely, let
\begin{equation}\label{eq-cfh0}
\mathcal{H}_0(r,\vartheta,p_r,p_\vartheta)=\displaystyle \dfrac{1}{2}\, \left(p_{r}^2+\dfrac{1}{r^2} \, p_{\vartheta}^2\right) - V(r)
\end{equation}
be the Hamiltonian associated with \eqref{eq-cenintro} in polar coordinates $x = r e^{i\vartheta}$. Then $p_r = \dot r$ and $p_\vartheta =  r^2 \dot \theta = L$, so that the radial dynamics is governed by the oscillator
\begin{equation*}
\dfrac{1}{2}\, \dot{r}^2+\dfrac{L^2}{2r^2}-V(r)=H.
\end{equation*}
With this in mind, and assuming that the above equation determines a closed orbit in the $(r,\dot r)$-plane, $T(H,L)$ and $\Theta(H,L)$ are defined respectively as
\begin{equation*}
T(H,L) = \sqrt{2}\, \int_{r_-(H,L)}^{r_+(H,L)} \dfrac{\mathrm{d}r}{\sqrt{H - L^2/2r^2 + V(r)}},
\end{equation*}
and 
\begin{equation*}
\Theta(H,L) = \sqrt{2}L \, \int_{r_-(H,L)}^{r_+(H,L)} \dfrac{\mathrm{d}r}{r^2\sqrt{H - L^2/2r^2 + V(r)}},
\end{equation*}
where $(r^{\pm}(H,L),0)$ are the intersection points of the orbit with the axis $\dot r = 0$. Such quantities represent, respectively, the period of the orbit in the $(r,\dot r)$-plane and the angular variation in the radial period (that is, the so-called apsidal angle) and they allow to express the non-degeneracy \eqref{eq-nondegintro} as
\begin{equation*}
\partial_H T(H^*,L^*) \partial_L \Theta(H^*,L^*) - \partial_L T(H^*,L^*) \partial_H \Theta(H^*,L^*) \neq 0,
\end{equation*}
where $(H^*,L^*)$ is the energy/angular momentum pair corresponding to the action value $I^*$ via the change of coordinates.
The advantage of this reformulation is that both $T$ and $\Theta$ can be regarded as time-maps of nonlinear oscillators (for the map $T$ this is direct, while for the map $\Theta$ a preliminary trick is needed, cf.~\cite{Ro-18}): using this fact, in \cite{BDF2} we showed that the non-degeneracy condition is satisfied for various potentials of physical interest, including the homogeneous potential $V(r)=\kappa/r^{\alpha}$ for $\kappa > 0$ and $\alpha\in(-\infty,2)\setminus\{-2,0,1\}$. By the use of the abstract bifurcation result in \cite{FoGaGi-16}, we then dealt with non-autonomous perturbed problems of the type
\begin{equation*}
\ddot x = V'(\vert x \vert) \frac{x}{\vert x \vert} + \varepsilon \, \nabla_x U(t,x),
\end{equation*}
where $U$ is an external potential $T$-periodically depending on the time variable $t$ (for a fixed $T > 0$), showing the existence
of $T$-periodic solutions bifurcating from tori filled by non-circular $T$-periodic solutions of the unperturbed problem. 
In particular, an application to a planar restricted 3-body problem with non-Newtonian interactions was obtained.

\subsection{Our contribution and plan of the paper}

In the present paper, we deal instead with autonomous Hamiltonian systems, looking for bifurcating periodic solutions at a fixed energy
level (independent from the parameter $\varepsilon$). 
More precisely, we consider perturbed Hamiltonians of the type
\begin{equation}\label{eq-heintro}
\mathcal{H}_{\varepsilon}(r,\vartheta,p_r,p_\vartheta) = \mathcal{H}_0(r,\vartheta,p_r,p_\vartheta) + \varepsilon \widetilde{\mathcal{H}} (r,\vartheta,p_r,p_\vartheta;\varepsilon),
\end{equation}
where $(r,\vartheta,p_r,p_\vartheta) \in \Xi \times \mathbb{T}^1 \times \mathbb{R}^2$, with $\Xi \subset (0,+\infty)$ an interval, and the unperturbed Hamiltonian is of the form 
\begin{equation} \label{eq-hamunpertastratta0}
\mathcal{H}_0(r,\vartheta,p_r,p_\vartheta)=\displaystyle \dfrac{1}{2}\, \left(\alpha(r)\, p_{r}^2+\dfrac{1}{r^2} \, p_{\vartheta}^2\right) - V(r),
\end{equation}
where $\alpha$ is a positive function and $V$ is a given potential, 
and we look for non-constant periodic solutions (i.e., closed orbits\footnote{In the literature, there is no an univocal definition of the term \emph{closed orbit}: often by this it is meant the image of a (non-constant) periodic solution, but sometimes the term is used more freely. Throughout the paper, we use closed orbit as a synonymous of non-constant periodic solution: we think however that the term orbit is useful to emphasize that we are dealing with autonomous dynamical systems.}) of the associated Hamiltonian system satisfying the energy condition
\begin{equation*}
\mathcal{H}_{\varepsilon}(r(t),\vartheta(t),p_r(t),p_\vartheta(t)) \equiv H^*, \quad \mbox{ for every } \varepsilon,
\end{equation*}
where $H^*$ is the energy of the solutions lying on the unperturbed torus from which bifurcation occurs.
Let us notice that unperturbed Hamiltonians of the form \eqref{eq-hamunpertastratta0} include in particular, for $\alpha \equiv 1$, 
the Hamiltonian \eqref{eq-cfh0} coming from central force problems, so that our result will apply to autonomous perturbed problems of the type
\begin{equation*}
\ddot x = V'(\vert x \vert) \frac{x}{\vert x \vert} + \varepsilon \, \nabla_x U(x),
\end{equation*}
where however, compared to \cite{BDF2}, bifurcating solutions preserve the energy rather than the period. Moreover, as discussed in detail below in the introduction, the class \eqref{eq-hamunpertastratta0} also allows us to deal with equatorial geodesics for the Schwarzschild metric of general relativity, which as already mentioned is the main motivation and application of the results in this paper.

In Section \ref{section-2}, as a preliminary step for our main investigation, we first provide 
a general result for bifurcation of closed orbits for nearly integrable Hamiltonian systems like \eqref{eq-hsintro} in the special case of two degrees of freedom. As recently shown in \cite{BDF4}, a result of this type can be obtained, whenever the non-degeneracy condition \eqref{eq-nondegintro2} holds true, by using an abstract bifurcation theory from periodic manifolds developed by Weinstein in a series of papers \cite{We-73,We-78} (this application was suggested by Weinstein himself, though, to the best of our knowledge, not rigorously proved).
Differently from \cite{BDF4}, here our proof is of purely dynamical nature and it goes as follows. At first, via an energetic reduction procedure (cf.~\cite[Section~45]{Ar-89}), the existence of periodic solutions to \eqref{eq-hsintro} with a fixed energy level
is converted, for $\varepsilon$ small, into the existence of periodic solutions for a non-autonomous Hamiltonian system with $1$ degree of freedom and, hence, to the existence of periodic points for an exact symplectic map $M_\varepsilon$ defined on a planar annulus. 
Then, we show that the non-degeneracy condition \eqref{eq-nondegintro2} implies a local twist condition for the map $M_\varepsilon$, so that, basically, the planar version of the Poincar\'e--Birkhoff theorem applies, providing the desired periodic points. 
Actually, up to a more stringent regularity assumption on the Hamiltonian, a joint use of KAM theory (precisely, Moser invariant curve theorem) and Poincar\'e--Birkhoff theorem can be made: in this way, for the same value of $\varepsilon$, an infinite number of closed orbits, bifurcating from the family of periodic tori closed enough to the one associated with action $I^*$ and with the same energy level,
can be obtained, similarly to a Birkhoff--Lewis scenario \cite{BiLe-34}. This provides a substantial improvement of the result in \cite{BDF4}, where bifurcation from a single torus was considered. See also Remark~\ref{remPBclassico} and Remark~\ref{rem1} for more comments.

In Section \ref{section-main}, we focus on the case of Hamiltonians of the form \eqref{eq-heintro}, with $\mathcal{H}_0$ as in \eqref{eq-hamunpertastratta0}, with the main aim of reformulating the non-degeneracy condition \eqref{eq-nondegintro2} in terms of time-maps, similarly to what was done in \cite{BDF2} for the condition \eqref{eq-nondegintro}. 
Actually, it turns out that only the apsidal angle $\Theta$, here defined as 
\begin{equation*}
\Theta(H,L) = \sqrt{2}L \, \int_{r_-(H,L)}^{r_+(H,L)} \dfrac{\mathrm{d}r}{r^2\sqrt{(H - L^2/2r^2 + V(r))\alpha(r)}},
\end{equation*}
plays a role: precisely, condition \eqref{eq-nondegintro2} can be written simply as
\begin{equation}\label{eq-nondegintrotheta}
\partial_L \Theta(H^*,L^*) \neq 0,
\end{equation}
where, again, $(H^*,L^*)$ is the energy/angular momentum pair corresponding to the action value $I^*$ via the change of coordinates.
The corresponding bifurcation result is precisely stated and proved in Theorem~\ref{teo-main}.

Taking advantage of some computations and considerations already developed in \cite{BDF2}, two immediate applications of this result to perturbations of central force problem are then given. The first one deal with perturbations of the Levi-Civita
equation 
\begin{equation*}
\ddot x = -\kappa \dfrac{x}{|x|^3} - 2\lambda \dfrac{x}{|x|^4}, \qquad \kappa,\lambda > 0,
\end{equation*}
proposed in \cite{LeCi-28} as a relativistic correction for the Kepler problem (see~\cite{BDF4} and the references therein): here the map $\Theta$ can be easily computed and condition \eqref{eq-nondegintrotheta} is directly checked to hold. 
The second application is, instead, for perturbations of the homogeneous equation
\begin{equation*}
\ddot x = -\kappa \dfrac{x}{|x|^{\alpha + 2}}, \qquad \kappa > 0, \; \alpha < 2,
\end{equation*}
where, we recall, the condition $\alpha < 2$ is needed since it is a necessary and sufficient one for the existence of non-circular closed orbits, cf.~\cite{AmCo-93,BDF2}. In this case, the map $\Theta(H,L)$ does not admit an explicit expression for a general value of $\alpha$, but, taking advantage of tools from the theory of time-maps, one can still prove that the non-degeneracy condition  \eqref{eq-nondegintrotheta} is satisfied for every $\alpha\in(-\infty,2)\setminus\{-2,1\}$. Incidentally, let us notice that the excluded cases $\alpha = 1$ and $\alpha = -2$ correspond, respectively, to the harmonic oscillator and to the Kepler problem: as well known, by Bertrand's theorem, these are the only two central force problems for which all bounded orbits are closed, and in both the cases the map $\Theta$ is actually constant, cf.~\cite{OrRo-19}.

In Section~\ref{section-sch} we finally present our main application, dealing with the geodesic dynamics of perturbed Schwarzschild problems.
The Schwarzschild problem is the relativistic analogue of the Kepler problem, and it consists in the study of the motion of a small point mass around a spherically symmetric blackhole, see \cite{MTW-73,Sc-84}. According to the principles of general relativity, such a motion is described by a so-called worldline in spacetime, that is a path $t \mapsto (\tau(t),z(t)) \in \mathbb{R}\times\mathbb{R}^3$, where $z \in \mathbb{R}^3$ is the position, $\tau$ is the coordinate time (that is, the time measured by a stationary clock at infinity) and $t$ is the proper time of the particle (that is, the time measured by a clock carried along with
the particle); moreover, this worldline is obtained as a timelike geodesic of a suitable Lorentzian metric defined on the spacetime.
Precisely, in spherical coordinates and geometrized units ($G=c=1$), the Schwarzschild metric $g$ is given by
\begin{equation*}
\mathrm{d}s^{2} = -\alpha(r)\, \mathrm{d}\tau^2+\alpha(r)^{-1}\, \mathrm{d}r^2+r^2\, \mathrm{d}\theta^2+r^2\sin^2\theta \, \mathrm{d}\phi^2,
\end{equation*}
where $r$ is the polar radius, $\theta$ is the inclination angle (colatitude) and $\phi$ is the azimuthal angle, and 
\begin{equation*}
\alpha(r)=1-\dfrac{2M}{r},
\end{equation*}
with $M>0$ the mass of the blackhole. 

We are interested in equatorial timelike closed geodesics for such a metric: a Hamiltonian formulation for this problem can be obtained as follows. Closed geodesics with respect to the metric are, by definition, non-constant periodic solutions of the Lagrangian system associated with the Lagrangian $L_0(q,\dot q) = \tfrac{1}{2} g_{\mu \nu} \dot{q}^\mu \dot{q}^\nu$, 
with $q = (\tau,r,\theta,\phi)$, and so by Legendre transformation the corresponding Hamiltonian is given by
\begin{equation*}
H_0 (q,p)= \dfrac{1}{2} \left(-\alpha(r)^{-1}\, p_{\tau}^2+\alpha(r)\, p_{r}^2+\dfrac{1}{r^2} \, p_{\theta}^2+\dfrac{1}{r^2\sin^2 \theta} \, p_{\phi}^2\right),
\end{equation*}
where $p=(p_{\tau},p_r,p_\theta,p_\phi)$ are the canonical momenta. Since such a Hamiltonian does not depend on $\tau$, we immediately see that $p_\tau = E$ is a first integral (the so-called particle energy); moreover, it is easy to check that the equatorial plane $\theta = \pi/2$ is invariant for the dynamics. Hence, the reduced Hamiltonian describing the equatorial geodesic dynamics is given by
\begin{equation*}
{\mathcal{H}_0} (r,\phi,p_r,p_\phi)= \dfrac{1}{2} \left(\alpha(r)\, p_{r}^2+\dfrac{1}{r^2} \, p_{\phi}^2\right)-\dfrac{1}{2}\, \dfrac{E^2}{\alpha(r)},
\end{equation*}
which is clearly of the form \eqref{eq-hamunpertastratta0} for $V(r) = E^2/2\alpha(r)$, with $r > 2M$, i.e., outside the so-called event horizon. Equatorial geodesic which are timelike, in particular, are found by imposing the isoenergetic condition
\begin{equation*}
\mathcal{H}_0 (r,\phi,p_r,p_\phi)=-\dfrac{1}{2}.
\end{equation*}
In the above framework, we first prove that for a (quite small!) interval of values for the first integral $E$, precisely $E \in (\sqrt{25/27},1)$, 
equatorial, timelike, bounded (i.e., radially closed) geodesics actually exist, for $L \in (L_E,4M)$, with $L_E > 0$ a suitable value. Unfortunately, however, the apsidal angle $\Theta$ does not admit an explicit expression, so that investigating the validity of the non-degeneracy condition
\begin{equation}\label{eq-nondeglast}
\partial_L \Theta\left(-\frac{1}{2},L^*\right) \neq 0,
\end{equation}
requires more work. Briefly, our strategy to tackle the problem is as follows. 
First, we use quite delicate arguments of the theory of time-maps to prove that
\begin{equation*}
\lim_{L \to (L_E)^-}\partial_L \Theta\left(-\frac{1}{2},L\right) \neq 0;
\end{equation*}
then, we use the analyticity of the map $\Theta$ to infer that the set of values $L^* \in (L_E,4M)$ for which there is an equatorial timelike closed geodesics satisfying \eqref{eq-nondeglast} is a dense subset of the interval $(L_E,4M)$.
At this point, Theorem~\ref{teo-main} can be applied to ensure bifurcation of equatorial timelike closed geodesics for a stationary and equatorial perturbation of the Schwarzschild metric. We refer to Theorem~\ref{teo-schwperturbato} for the precise statement.

\section{Bifurcation of closed orbits for nearly integrable Hamiltonian systems with two degrees of freedom}\label{section-2}
	
In this section, we state and prove a preliminary result ensuring, for a nearly integrable Hamiltonian system with two degrees of freedom, bifurcation of closed orbits from periodic invariant tori of the unperturbed problem. More precisely, in Section~\ref{section-2.1} we first discuss a Poincar\'{e}--Birkhoff-type theorem for exact symplectic maps in the plane; then, in Section~\ref{section-2.2} we show how to apply this result to the case of nearly integrable Hamiltonian systems with two degrees of freedom, via an energetic reduction procedure. 
	
\subsection{Periodic points of exact symplectic maps in the plane}\label{section-2.1}

Let us consider a family of maps $M_{\varepsilon} \colon \mathbb{R} \times \mathopen{[}a,b\mathclose{]} \to \mathbb{R}^{2}$ of class $\mathcal{C}^{4}$ (the index $\varepsilon$ denotes a small real parameter, precisely $|\varepsilon|\leq \bar{\varepsilon}$), having the form
\begin{equation}\label{forma}
M_{\varepsilon}(\theta,r) = \bigl{(} \theta + \alpha(r) + \mu_{1}(\theta,r;\varepsilon), r + \mu_{2}(\theta,r;\varepsilon) \bigr{)},
\end{equation}
where $\mu_{1}$ and $\mu_{2}$ are $2\pi$-periodic in the variable $\theta$ and satisfy
$\mu_{1}(\theta,r;0) = \mu_{2}(\theta,r;0) \equiv 0$. We also assume that each $M_{\varepsilon}$ is \textit{exact symplectic}, in the sense that there exists a function $V_{\varepsilon} = V_{\varepsilon}(\theta,r) \colon \mathbb{R} \times \mathopen{[}a,b\mathclose{]} \to \mathbb{R}$ of class $\mathcal{C}^{4}$, $2\pi$-periodic in $\theta$, and such that
\begin{equation}\label{exact}
\mathrm{d}V_{\varepsilon} = M_{\varepsilon}^{*} \lambda - \lambda, \quad \text{where $\lambda = r \, \mathrm{d}\theta$.}
\end{equation}

We are interested in the existence of periodic points of the map $M_{\varepsilon}$. Precisely, given $(m_{1},m_{2}) \in (\mathbb{N}\setminus\{0\}) \times \mathbb{Z}$, we say that a point $(\bar\theta,\bar r)$ is a $(m_{1},m_{2})$-periodic point for the map $M_{\varepsilon}$ if $(\bar\theta,\bar r) \in \mathrm{dom}(M_{\varepsilon}^{j})$, for every $j = 1,\ldots,m_1$, and 
\begin{equation*}
M_{\varepsilon}^{m_1}(\bar{\theta},\bar r) = (\bar\theta + 2\pi m_2, \bar r).
\end{equation*}
Let us notice that if $(\bar{\theta},\bar r)$ is a $(m_1,m_2)$-periodic point, then the points $M_{\varepsilon}^j(\bar{\theta},\bar r)$ are also 
$(m_1,m_2)$-periodic points, for every $j = 1,\ldots,m_1-1$: all these points have to be considered as equivalent (they form a so-called periodicity class). Moreover, if  $(\bar\theta,\bar r)$ is a $(m_1,m_2)$-periodic point, then all the points of the type
$(\bar\theta + 2\pi k,\bar r)$ with $k \in \mathbb{Z}$ are also $(m_1,m_2)$-periodic points, and again such points have to be considered equivalent.
Hence, in what follows we agree to say that two $(m_1,m_2)$-periodic points $(\bar\theta_1,\bar r_1)$ and $(\bar\theta_2,\bar r_2)$ are distinct if 
\begin{equation}\label{formula-distinti1}
M_{\varepsilon}^j(\bar\theta_1,\bar r_1) - (\bar\theta_2,\bar r_2) \notin 2\pi \mathbb{Z} \times \{0\}, \quad \text{for every $j \in\{1,\ldots,m_1-1\}$.} 
\end{equation}

The next result ensures the existence of a countable family of periodic points for the map $M_{\varepsilon}$, provided the parameter $\varepsilon$ is small enough and a local twist condition is assumed. We believe that a result of this type, which can be obtained by a combined use of KAM theory and Poincar\'e--Birkhoff theorem, is well-known to the experts of the field (see for instance \cite{DeTe-23,Or-96} for two papers developing, in different contexts, a similar procedure). 
However, since it seems difficult to find an adequate reference in the literature, we provide a statement which is suited for the applications we have in mind, together with a sketch of the proof. See also Remark~\ref{remPBclassico} and Remark~\ref{remAM} for further comments.

\begin{theorem}\label{KAMePB}
Let $M_{\varepsilon} \colon \mathbb{R} \times [a,b] \to \mathbb{R}^2$ be a one-to-one map of class $\mathcal{C}^4$, exact symplectic, and satisfying \eqref{forma}. 
Moreover, we assume that
\begin{equation}\label{resto}
\Vert \mu_1(\cdot,\cdot; \varepsilon) \Vert_{\mathcal{C}^4} + \Vert \mu_2(\cdot,\cdot; \varepsilon) \Vert_{\mathcal{C}^4} \to 0, \quad \text{as $\varepsilon \to 0$.}
\end{equation}
Furthermore, let us suppose that for some $r^* \in \mathopen{]}a,b\mathclose{[}$ there exists a pair of coprime\footnote{
We recall that, in the case $n_2 = 0$, this implies that $n_1 = 1$.} integers $(n_1,n_2) \in (\mathbb{N}\setminus\{0\}) \times \mathbb{Z}$ such that
\begin{equation}\label{condizione1}
\alpha(r^*) = 2\pi \frac{ n_2}{n_1}.
\end{equation}
Finally, let us assume
\begin{equation}\label{condizione2}
\alpha'(r^*) \neq 0.
\end{equation}
Then, there exist $\varepsilon^* > 0$ and $\delta^* > 0$ such that, for every pair of coprime integers $(m_1,m_2) \in (\mathbb{N}\setminus\{0\}) \times \mathbb{Z}$ satisfying
\begin{equation*}
\left\vert \frac{m_2}{m_1} - \frac{n_2}{n_1}\right \vert < \delta^*,
\end{equation*}
and for every $\varepsilon$ satisfying $\vert\varepsilon \vert < \varepsilon^*$, the map $M_{\varepsilon}$ has at least two
distinct $(m_1,m_2)$-periodic points. 
\end{theorem}

\begin{remark}\label{rem-2.1}
We point out that the injectivity of the map $M_{\varepsilon}$ follows from assumptions \eqref{forma} and \eqref{resto}, when $\varepsilon$ is small enough: therefore, this hypothesis could be dropped. However, since it will be directly satisfied in the application of Theorem~\ref{KAMePB} given in Section~\ref{section-2.2}, we have preferred to keep it for the sake of clarity.
\hfill$\lhd$
\end{remark}

\begin{remark}\label{remPBclassico}
We stress that the crucial feature of Theorem~\ref{KAMePB} lies in the fact that, for $\varepsilon$ small enough, an infinite (countable) family of periodic points (precisely, $(m_1,m_2)$-periodic points with $m_2/m_1$ close enough to $n_2/n_1$) is provided.

In order to clarify this issue, let us observe that the existence of a $(n_1,n_2)$-periodic point, for $\vert \varepsilon \vert < \varepsilon^*$, could be proved by a simple application of the Poincar\'e--Birkhoff theorem for non-invariant annuli (see \cite{Di-82,Fr-88,LCWa-10,Re-97}): to this end, much weaker regularity assumptions on $M_{\varepsilon}$ are needed (in particular, it is enough that $M_{\varepsilon}$ is a map of class $\mathcal{C}^1$ which is close to $M_0$ in the $\mathcal{C}^0$-norm). Of course, since the twist condition \eqref{condizione2} extends by continuity to a neighborhood of $r^*$, one can repeat the argument to find $(m_1,m_2)$-periodic points with $m_2/m_1$ close enough to $n_2/n_1$. However, this requires in principle to change the value of $\varepsilon^*$, so that a uniform choice of $\varepsilon^*$ is possible only for a finite (yet arbitrarily large) number of $(m_2,m_1)$. See for instance \cite{FoSaZa-12} for a very general result in this spirit, dealing however with periodic solutions of planar Hamiltonian systems rather than periodic points of planar maps.

As it will be clear from the proof, the strategy to find infinitely many periodic points for the same value of $\varepsilon$ is a joint use of KAM theory (precisely, Moser invariant curve theorem) and Poincar\'e--Birkhoff theorem. More precisely, one first finds, for $\vert \varepsilon \vert < \varepsilon^*$, an invariant annulus (containing the circle $\{r = r^*\}$) for the map $M_{\varepsilon}$; then, the Poincar\'e--Birkhoff theorem is applied on this annulus (the crucial point is that, since the annulus is now invariant, the boundary twist condition extends to all the iterates of $M_{\varepsilon}$). Of course, this requires more severe regularity assumptions on $M_{\varepsilon}$ and, in particular, condition
\eqref{resto} to hold true.

It is interesting to mention that a result in this spirit is contained in \cite{BeKa-87}: such a result holds true also in higher dimension,  however it requires an extra convexity assumption on the generating function of the map, which in our case would lead to $\alpha'<0$.
\hfill$\lhd$
\end{remark}

\begin{proof}[Sketch of the proof of Theorem~\ref{KAMePB}]
We give the proof in the case $\alpha'(r^*) > 0$, the complementary situation being analogous. 
Correspondingly, we take $\eta > 0$ so that $J = [r^*-\eta,r^*+\eta] \subset [a,b]$ and $\alpha'(r) > 0$ for every $r \in J$. Then, thanks to \eqref{resto} and recalling that assumption \eqref{exact} implies the so-called intersection property (see, for instance, \cite[Exercise~6]{KuOr-13} and \cite[Exercise~13]{Or-01}), we can apply to the map $M_{\varepsilon}\colon \mathbb{R} \times J \to \mathbb{R}^2$ the Moser invariant curve theorem in the version by Herman \cite{He-83,He-86}.
Hence, there exists a value $\varepsilon^* > 0$ such that, for $\vert \varepsilon \vert < \varepsilon^*$, 
there are two functions $r_i^\varepsilon \colon \mathbb{R} \to \mathbb{R}$, of class at least $\mathcal{C}^1$ and $2\pi$-periodic, such that, for $i=1,2$,
\begin{equation*}
M_{\varepsilon}(\Gamma_i^\varepsilon) = \Gamma_i^\varepsilon, 
\quad \text{where $\Gamma_i^\varepsilon = \{(\theta,r_i^\varepsilon(\theta) \colon \theta \in \mathbb{R} \}.$}
\end{equation*}
By the invariant curve theory (cf.~\cite[Remark, p.~3]{Mo-62}), these curves can be taken in such a way that the rotation number of $M_{\varepsilon} \colon \Gamma_i^\varepsilon \to \Gamma_i^\varepsilon$ is equal to a value $\rho_i$ independent of $\varepsilon$, with
\begin{equation*}
\rho_1 < \frac{n_2}{n_1} < \rho_2,
\end{equation*}
and, moreover, that
\begin{equation*}
r^*-\eta < \min_{\theta}r_1^\varepsilon(\theta) \leq \max_{\theta}r_1^\varepsilon(\theta) < r^* < \min_{\theta}r_2^\varepsilon(\theta) \leq \max_{\theta}r_2^\varepsilon(\theta) < r^* + \eta.
\end{equation*}
Now, we consider the region
\begin{equation*}
\mathcal{A}_{\varepsilon} = \{ (\theta,r) \colon \theta \in \mathbb{R}, \,r_1^\varepsilon(\theta) \leq r \leq r_2^\varepsilon(\theta)\}.
\end{equation*} 
It is easy to check that $\mathcal{A}_{\varepsilon}$ is invariant for $M_{\varepsilon}$ and that
$M_{\varepsilon} \colon \mathcal{A}_{\varepsilon} \to \mathcal{A}_{\varepsilon}$ is a (symplectic) homeomorphism.
Moreover, it is well-known (see, for instance, \cite[Remark~5.15]{DeTe-23}) that there exists a symplectic change of variable 
$\Phi_{\varepsilon}$ transforming the curves $\Gamma_i^\varepsilon$ into straight lines $\mathbb{R} \times \{c_i^\varepsilon\}$
and the region $\mathcal{A}_{\varepsilon}$ into the standard strip $\mathbb{R} \times [c_1^\varepsilon,c_2^\varepsilon]$. Accordingly, we can consider the symplectic homeomorphism
\begin{equation*}
\widetilde{M}_{\varepsilon} = \Phi_{\varepsilon} \circ M_{\varepsilon} \circ \Phi_{\varepsilon}^{-1} \colon \mathbb{R} \times [c_1^\varepsilon,c_2^\varepsilon] \to \mathbb{R} \times [c_1^\varepsilon,c_2^\varepsilon]. 
\end{equation*}
This construction preserves the rotation numbers and so, in particular, the rotation number of $\widetilde{M}_{\varepsilon}\colon \mathbb{R} \times \{c_i^\varepsilon\} \to \mathbb{R} \times \{c_i^\varepsilon\}$ is $\rho_i$. Therefore, the Poincar\'e--Birkhoff theorem in the version \cite[Theorem~7.1]{Go-01}
can be applied, ensuring the existence of two distinct $(m_1,m_2)$-periodic points for $\widetilde{M}_{\varepsilon}$ whenever 
$m_2/m_1 \in (\rho_1,\rho_2)$
and thus, in particular, when
\begin{equation*}
\left\vert \frac{m_2}{m_1} - \frac{n_2}{n_1} \right\vert < \delta^* := \min_{i=1,2} \left\vert \rho_i - \frac{n_2}{n_1} \right\vert. 
\end{equation*}
Undoing the change of variable, we thus find two distinct $(m_1,m_2)$-periodic points for the original map $M_{\varepsilon}$.
\end{proof}

\begin{remark}\label{remAM}
Looking at the above proof, it is immediately realized that, after having applied KAM theory to find an invariant annulus, it would be possible to use Aubry--Mather theory (see, for instance, \cite[Theorem~8.3]{Go-01}) to find, besides periodic points, Mather sets associated with irrational rotation numbers. We do not explore further this possibility, since it is not clear the role played by such sets in the application developed in the next section. This will be the subject of future investigations.
\hfill$\lhd$
\end{remark}

\subsection{The bifurcation result}\label{section-2.2}
	
Let us now consider a Hamiltonian system with two degrees of freedom of the form
\begin{equation} \label{eq-sispertazioneangolo}
\begin{cases}
\, \dot{\varphi}=\nabla \mathcal{K}_0 (I)+\varepsilon \, \nabla_I \widetilde{\mathcal{K}}(\varepsilon,\varphi,I),
\vspace{3pt}\\
\, \dot{I}=-\varepsilon \, \nabla_\varphi \widetilde{\mathcal{K}}(\varepsilon,\varphi, I), 
\end{cases}
\end{equation}
where $\mathcal{K}_0 \colon D\to \mathbb{R}$ is a function of class $\mathcal{C}^6$, with $D\subset \mathbb{R}^2$ open, and $\widetilde{\mathcal{K}} \colon (-\hat\varepsilon,\hat\varepsilon)\times \mathbb{R}^2 \times D\to \mathbb{R}$ is a function of class $\mathcal{C}^6$ which is $2\pi$-periodic in each $\varphi_i$-variable, where $\varphi = (\varphi_1,\varphi_2)$.
For every $\varepsilon \in (-\hat\varepsilon,\hat\varepsilon)$, let 
\begin{equation*}
\mathcal{K}_{\varepsilon} (\varphi,I)= \mathcal{K}_0(I)+\varepsilon \, \widetilde{\mathcal{K}}(\varepsilon,\varphi, I),
\quad (\varphi,I)\in \mathbb{R}^2 \times D,
\end{equation*}
be the Hamiltonian of the system \eqref{eq-sispertazioneangolo}.
	
We point out that system \eqref{eq-sispertazioneangolo} has to be meant as the lifting of a corresponding Hamiltonian system on $\mathbb{T}^2\times D$. Accordingly, from now on we agree to say that a solution $(\varphi,I)$ of \eqref{eq-sispertazioneangolo} is $T$-periodic, for some $T>0$, if $I$ is $T$-periodic and
\begin{equation*}
\varphi(T)-\varphi(0)\in 2\pi \mathbb{Z}^2.
\end{equation*}
More precisely, if 
\begin{equation*}
\varphi(T)-\varphi(0) = (2\pi l_1,2\pi l_2),
\end{equation*}
then we say that $(\varphi,I)$ is a $T$-periodic solution with winding vector $(l_1,l_2) \in \mathbb{Z}^2$. 
Incidentally, let us observe that if $l_1$ and $l_2$ are coprime (in particular both non-zero), then $T$ is necessarily the minimal period of the solution. 
	
We also notice that if $(\varphi(t),I(t))$ is a $T$-periodic solution, then $(\varphi(t + \tau),I(t+\tau))$ is a $T$-periodic solution for every $\tau \in \mathopen{[}0,T\mathclose{[}$; moreover, $(\varphi(t) + 2\pi k, I(t))$ with $k \in \mathbb{Z}$ is a $T$-periodic solution, as well. Accordingly, from now on we agree to say that two $T$-periodic solutions $(\varphi^1,I^1)$ and $(\varphi^2,I^2)$ of \eqref{eq-sispertazioneangolo} are distinct if 
\begin{equation}\label{formula-distinti2}
(\varphi^1(t+\tau),I^1(t+\tau)) - (\varphi^2(t),I^2(t)) \notin 2\pi \mathbb{Z}^2 \times \{(0,0)\}, \quad \text{for every $\tau \in \mathopen{[}0,T\mathclose{[}$ and $t \in \mathbb{R}$.} 
\end{equation}

Our aim is to prove the existence of periodic solutions (i.e.~closed orbits) with prescribed energy of 
\eqref{eq-sispertazioneangolo}, bifurcating from periodic solutions of the unperturbed system 
\begin{equation} \label{eq-sisazioneangolo}
\begin{cases}
\, \dot{\varphi}=\nabla \mathcal{K}_0 (I),
\\ 
\, \dot{I}=0.
\end{cases}
\end{equation}
To this end, let us first observe that the Hamiltonian system \eqref{eq-sisazioneangolo} is integrable; the solution $(\varphi(\cdot;\varphi^*,I^*),I(\cdot;\varphi^*,I^*))$ satisfying $(\varphi (0;\varphi^*,I^*),I(0;\varphi^*,I^*))=(\varphi^*,I^*)$, with $(\varphi^*,I^*)=(\varphi^*_{1},\varphi^*_{2},I^*_{1},I^*_{2})\in \mathbb{R}^2\times D$, is given by
\begin{equation} \label{eq-solazioneangolo}
(\varphi(t;\varphi^*,I^*),I(t;\varphi^*,I^*))=(\varphi^*+t\, \nabla \mathcal{K}_0(I^*),I^*),
\quad \text{for all $t\in \mathbb{R}$.}
\end{equation}
Let us now assume that
\begin{equation*}
\partial_{I_1} \mathcal{K}_0(I^*) \neq 0
\end{equation*}
and that there exists a pair of coprime integer numbers $(n_1,n_2) \in (\mathbb{N}\setminus\{0\}) \times \mathbb{Z}$ such that 
\begin{equation} \label{eq-ipoazioneangolo}
\frac{\partial_{I_2}\mathcal{K}_0(I^*)}{\partial_{I_1}\mathcal{K}_0(I^*)} = \frac{n_2}{n_1}.
\end{equation}
Defining 
\begin{equation*}
\tau^* = \frac{2\pi n_1}{\vert \partial_{I_1}\mathcal{K}_0(I^*) \vert},
\end{equation*}
from \eqref{eq-solazioneangolo} we thus have that for every $\varphi^*\in \mathbb{R}^2$ the solution $(\varphi(\cdot;\varphi^*,I^*),I(\cdot;\varphi^*,I^*))$ of \eqref{eq-sisazioneangolo} is constant in the $I$-component and satisfies
\begin{equation*}
\varphi(\tau^*;\varphi^*,I^*)=\varphi^*+ \mathrm{sgn} (\partial_{I_1}\mathcal{K}_0(I^*))(2\pi n_1,2\pi n_2),
\end{equation*}
that is, the invariant torus $\{I = I^*\}$ is filled by periodic solutions of minimal period $\tau^*$ and winding vector equal to $\mathrm{sgn} (\partial_{I_1}\mathcal{K}_0(I^*))(2\pi n_1,2\pi n_2)$.

The following result provides periodic solutions of \eqref{eq-sispertazioneangolo} bifurcating from the invariant torus $\{I = I^*\}$ and having prescribed energy $\mathcal{K}_{\varepsilon}=\mathcal{K}_0(I^*)$ (see also Remark~\ref{rem1} and Remark~\ref{rem2} for some comments).

\begin{theorem}\label{teo-pbapplicato}
Let $I^*=(I_1^*,I_2^*)\in D$ be such that \eqref{eq-ipoazioneangolo} is satisfied for a pair of coprime integers $(n_1, n_2)\in (\mathbb{N}\setminus\{0\})\times \mathbb{Z}$, and let $K^*=\mathcal{K}_0(I^*)$. Moreover, assume that
\begin{equation} \label{eq-iponondegenere}
\langle \nabla^2 \mathcal{K}_0(I^*) (-\partial_{I_2} \mathcal{K}_0(I^*), \partial_{I_1} \mathcal{K}_0(I^*)),(-\partial_{I_2} \mathcal{K}_0(I^*), \partial_{I_1} \mathcal{K}_0(I^*)) \rangle \neq 0.
\end{equation}
Then, there exist $\varepsilon^* >0$ and $\delta^* > 0$ such that, for every pair of coprime integers $(m_1,m_2) \in (\mathbb{N}\setminus\{0\}) \times \mathbb{Z}$ satisfying
\begin{equation}\label{hpm1}
\left\vert \frac{m_2}{m_1} - \frac{n_2}{n_1}\right \vert < \delta^*,
\end{equation}
and for every $\varepsilon$ satisfying $\vert\varepsilon \vert < \varepsilon^*$, there are two distinct periodic solutions of the Hamiltonian system \eqref{eq-sispertazioneangolo}, with energy $\mathcal{K}_{\varepsilon} = K^*$ and winding 
vector on their minimal period equal to  $(l_1,l_2) = \mathrm{sgn}(\partial_{I_1}\mathcal{K}_0(I^*)) (m_1,m_2)$. 
\end{theorem}
	
\begin{remark}\label{rem1}
Let us observe that the non-degeneracy condition \eqref{eq-iponondegenere} is equivalent to 
\begin{equation*}
\mathrm{det} \,
\begin{pmatrix}
\nabla^2 \mathcal{K}_0(I^*) & \nabla \mathcal{K}_0(I^*)^{\intercal} 
\vspace{3pt} \\
\nabla \mathcal{K}_0(I^*) & 0
\end{pmatrix} 
\neq 0,
\end{equation*}
which is the condition used in KAM theory for the preservation of quasi-periodic invariant tori at a fixed energy level (cf.~\cite[Appendix~8]{Ar-89} as well as \cite{BrHu-91} and the references therein): Theorem~\ref{teo-pbapplicato} can thus be seen as a periodic counterpart of isoenergetic KAM theory (for the case of two degrees of freedom). 
A previous result in this spirit, valid for the general case of $n$ degrees of freedom but providing bifurcation of solutions only for the choice $(m_1,m_2) = (n_1,n_2)$, has been recently proved in \cite{BDF4}, using a completely different approach, that is, Weinstein's bifurcation theory from periodic manifolds \cite{We-73,We-78}. We stress that the main difference between \cite[Theorem~2.2]{BDF4} and Theorem~\ref{teo-pbapplicato} lies in the fact that, for the same value of $\varepsilon$, 
an infinite (countable) family of periodic solutions, labeled by the pair of integers $(m_1,m_2)$, is provided, cf.~Remark~\ref{remPBclassico}.
\hfill$\lhd$
\end{remark}

The rest of this section is devoted to the proof of Theorem~\ref{teo-pbapplicato}, which will be achieved by  first performing an energetic reduction so as to reduce the degrees of freedom to one (cf.~\cite[Section~45]{Ar-89}), and then applying Theorem~\ref{KAMePB}
to the Poincar\'e map of the reduced system.
	
So, let us proceed at first with the energetic reduction. From the first relation in \eqref{eq-ipoazioneangolo}, by a semiglobal version of the implicit function theorem we deduce that there exist  $\varepsilon^* \in (0,\hat\varepsilon)$, $[a,b]\subset \mathbb{R}$, with $I_2^*\in [a,b]$, and a unique function $\mathcal{I} \colon [-\bar{\varepsilon},\bar{\varepsilon}]\times \mathbb{R}^2\times [a,b]\to \mathbb{R}$ such that
\begin{equation*}
\mathcal{K}_{\varepsilon} (\varphi_1,\varphi_2,\mathcal{I}(\varepsilon,\varphi_1,\varphi_2,I_2),I_2)=K^*,
\quad (\varphi_1,\varphi_2,I_2)\in \mathbb{R}^2\times [a,b].
\end{equation*}
Notice that the periodicity of $\mathcal{K}_{\varepsilon}$ in the pair $(\varphi_1,\varphi_2)$ implies that $\mathcal{I}$ is $2\pi$-periodic in each of the variables $\varphi_1$ and $\varphi_2$.
Moreover, we have $\mathcal{I}\in \mathcal{C}^6 ([-\bar{\varepsilon},\bar{\varepsilon}]\times \mathbb{R}^2\times [a,b])$ and 
\begin{align} 
\partial_{\varphi_1} \mathcal{I}(\varepsilon,\varphi_1,\varphi_2,I_2)=-\dfrac{\partial_{\varphi_1} \mathcal{K}_{\varepsilon} (\varphi_1,\varphi_2, \mathcal{I}(\varepsilon,\varphi_1,\varphi_2,I_2),I_2)}{\partial_{I_1} \mathcal{K}_{\varepsilon} (\varphi_1,\varphi_2, \mathcal{I}(\varepsilon,\varphi_1,\varphi_2,I_2),I_2)},
\\
\partial_{\varphi_2} \mathcal{I}(\varepsilon,\varphi_1,\varphi_2,I_2)=-\dfrac{\partial_{\varphi_2} \mathcal{K}_{\varepsilon} (\varphi_1,\varphi_2, \mathcal{I}(\varepsilon,\varphi_1,\varphi_2,I_2),I_2)}{\partial_{I_1} \mathcal{K}_{\varepsilon} (\varphi_1,\varphi_2, \mathcal{I}(\varepsilon,\varphi_1,\varphi_2,I_2),I_2)},
\\
\partial_{I_2} \mathcal{I}(\varepsilon,\varphi_1,\varphi_2,I_2)=-\dfrac{\partial_{I_2} \mathcal{K}_{\varepsilon} (\varphi_1,\varphi_2, \mathcal{I}(\varepsilon,\varphi_1,\varphi_2,I_2),I_2)}{\partial_{I_1} \mathcal{K}_{\varepsilon} (\varphi_1,\varphi_2, \mathcal{I}(\varepsilon,\varphi_1,\varphi_2,I_2),I_2)}, \label{eq-derfunzimpl}
\end{align}
for every $(\varepsilon,\varphi_1,\varphi_2,I_2)\in [-\bar{\varepsilon},\bar{\varepsilon}]\times \mathbb{R}^2\times [a,b]$. 
For further convenience we also observe that, since $\mathcal{K}_0$ does not depend on the variables $\varphi_1$ and $\varphi_2$, 
the same is true for the function $\mathcal{I}(0,\cdot,\cdot,\cdot)$ and so we define
\begin{equation}\label{eq-I0}
\mathcal{I}_0(I_2) = \mathcal{I}(0,\varphi_1,\varphi_2,I_2),
\quad I_2\in\mathopen{[}a,b\mathclose{]}.
\end{equation}
Moreover,
\begin{equation}\label{eq-I0*}
\mathcal{I}_0(I_2^*) = I_1^*.
\end{equation}
For every $\varepsilon\in \mathbb{R}$, with $|\varepsilon|\leq \bar{\varepsilon}$, we now define $\mathcal{K}^{\rm red}_{\varepsilon} \colon \mathbb{R}^2\times [a,b]\to \mathbb{R}$ by
\begin{equation} \label{eq-defkridotta} 
\mathcal{K}^{\rm red}_{\varepsilon} (\varphi_1,\varphi_2,I_2)=-\mathcal{I}(\varepsilon,\varphi_1,\varphi_2,I_2),
\end{equation}
for every $(\varphi_1,\varphi_2,I_2)\in \mathbb{R}^2\times [a,b]$. As previously observed, $\mathcal{K}^{\rm red}_{\varepsilon}$ is $2\pi$-periodic in the variables $\varphi_1$ and $\varphi_2$. 
The variable $\varphi_1$ is considered as a new time-variable; denoting by ${}'$ the derivative with respect to this variable, the Hamiltonian system associated with the Hamiltonian \eqref{eq-defkridotta} is 
\begin{equation} \label{eq-sisridotto}
\begin{cases}
\, \varphi'_2 = -\partial_{I_2} \mathcal{I} (\varepsilon,\varphi_1,\varphi_2,I_2),
\\
\, I'_2 = \partial_{\varphi_2} \mathcal{I} (\varepsilon,\varphi_1,\varphi_2,I_2).
\end{cases}
\end{equation}
Since this system is $2\pi$-periodic in the time-variable $\varphi_1$, it is natural to look for $2\pi m_1$-periodic solutions, with $m_1 \in (\mathbb{N}\setminus\{0\})$. More precisely, we say that a global solution $(\widetilde{\varphi_2},\widetilde{I_2})$ of \eqref{eq-sisridotto} is
$2\pi m_1$-periodic with winding number $m_2 \in \mathbb{Z}$ if $\widetilde{I_2}$ is $2\pi m_1$-periodic and 
\begin{equation*}
\widetilde{\varphi_2}(2\pi m_1) - \widetilde{\varphi_2}(0)= m_2.
\end{equation*}
Moreover we say that two such solutions $(\widetilde{\varphi_2}^1,\widetilde{I_2}^1)$ and $(\widetilde{\varphi_2}^2,\widetilde{I_2}^2)$ are distinct if
\begin{equation}\label{formula-distinti3}
\begin{aligned}
(\widetilde{\varphi_2}^1(t+2\pi m),\widetilde{I_2}^1(t+2\pi m))
&- (\widetilde{\varphi_2}^2(t),\widetilde{I_2}^2(t)) \notin
2\pi\mathbb{Z} \times \{0\}, 
\\
&\text{for every $m \in \{1,\ldots,m_1-1\}, t \in \mathbb{R}$.}
\end{aligned}
\end{equation}
The next lemma, which is the core of the energetic reduction procedure, shows that $2\pi m_1$-periodic solutions of 
\eqref{eq-sisridotto} gives rise to periodic solutions of \eqref{eq-sispertazioneangolo} having energy $K^*$.
	
\begin{lemma} \label{lem-solridotto}
Let $(\widetilde{\varphi_2},\widetilde{I_2})$ be a $2\pi m_1$-periodic solution of \eqref{eq-sisridotto} with winding number $m_2$, for 
a pair of coprime integers $(m_1,m_2) \in (\mathbb{N}\setminus\{0\}) \times \mathbb{Z}$. Define 
\begin{equation*}
\widetilde{\mathcal{I}_{\varepsilon}} (\varphi_1)=\mathcal{I}(\varepsilon,\varphi_1, \widetilde{\varphi_2} (\varphi_1),\widetilde{I_2}(\varphi_1)),
\quad \varphi_1 \in \mathbb{R},
\end{equation*}
and 
\begin{equation*}
t_{\varepsilon} (\varphi_1)=\int_0^{\varphi_1} \dfrac{1}{\partial_{I_1} \mathcal{K}_{\varepsilon} (\xi,\widetilde{\varphi_2}(\xi), \widetilde{\mathcal{I}_{\varepsilon}}(\xi),\widetilde{I_2}(\xi))}\, \mathrm{d}\xi,
\quad \varphi_1 \in \mathbb{R}.
\end{equation*}
Moreover, let $\widetilde{\varphi_1}=\widetilde{\varphi_1}(\cdot)$ be the inverse of ${t}_{\varepsilon}$ and 
\begin{equation*}
T_{\varepsilon} ={t}_{\varepsilon}(2\pi m_1).
\end{equation*}
Then, the function $(\widehat{\varphi},\widehat{I}) = (\widehat{\varphi_1},\widehat{\varphi_2},\widehat{I_1},\widehat{I_2})$ defined by
\begin{equation}\label{eq-lem}
(\widehat{\varphi_1}(t),\widehat{\varphi_2}(t),\widehat{I_1}(t),\widehat{I_2}(t))
=
(\widetilde{\varphi_1}(t),\widetilde{\varphi_2}(\widetilde{\varphi_1}(t)),\widetilde{\mathcal{I}_{\varepsilon}}(\widetilde{\varphi_1}(t)),\widetilde{I_2}(\widetilde{\varphi_1}(t))),
\quad t\in\mathbb{R},
\end{equation}
is a periodic solution of \eqref{eq-sispertazioneangolo}, with energy $K^*$,  minimal period equal to $\vert T_{\varepsilon} \vert$ and winding vector on the minimal period equal to
\begin{equation*}
(l_1,l_2) = \mathrm{sgn}(\partial_{I_1}\mathcal{K}_0(I^*)) (m_1,m_2).
\end{equation*}
Moreover,  if $(\widetilde{\varphi_2}^1,\widetilde{I_2}^1)$ and $(\widetilde{\varphi_2}^2,\widetilde{I_2}^2)$ are distinct solutions of \eqref{eq-sisridotto} in the sense of \eqref{formula-distinti3}
then the corresponding solutions $(\widehat{\varphi}^1,\widehat{I}^1)$ and $(\widehat{\varphi}^2,\widehat{I}^2)$ of 
\eqref{eq-sispertazioneangolo} are distinct in the sense of \eqref{formula-distinti2}.
\end{lemma}
	
\begin{proof}
We first prove that the function $(\widehat{\varphi_1},\widehat{\varphi_2},\widehat{I_1},\widehat{I_2})$ defined in \eqref{eq-lem} is a solution of system \eqref{eq-sispertazioneangolo}. Accordingly, we first observe that
\begin{align*}
\dot{\widehat{\varphi_1}}(t) 
= \dot{\widetilde{\varphi_1}}(t) 
= \dfrac{1}{t'(\widetilde{\varphi_1}(t))}
&= \partial_{I_{1}} \mathcal{K}_{\varepsilon} (\widetilde{\varphi_1}(t),\widetilde{\varphi_2}(\widetilde{\varphi_1}(t)),\widetilde{\mathcal{I}_{\varepsilon}}(\widetilde{\varphi_1}(t)),\widetilde{I_2}(\widetilde{\varphi_1}(t)))
\\
&= \partial_{I_{1}} \mathcal{K}_{\varepsilon} (\widehat{\varphi_1}(t),\widehat{\varphi_2}(t),\widehat{I_1}(t),\widehat{I_2}(t)),
\end{align*}
for all $t$. Next, by exploiting the above formula and the implicit function theorem, for all $t$, we have
\begin{align*}
\dot{\widehat{\varphi_2}}(t) 
&= \widetilde{\varphi_2}'(\widetilde{\varphi_1}(t)) \dot{\widetilde{\varphi_1}}(t)
= - \partial_{I_{2}} \mathcal{I}(\varepsilon,\widetilde{\varphi_1}(t),\widetilde{\varphi_2}(\widetilde{\varphi_1}(t)),\widetilde{I_2}(\widetilde{\varphi_1}(t))) 
\cdot
\dot{\widetilde{\varphi_1}}(t)
\\
&= - \partial_{I_{2}} \mathcal{I}(\varepsilon,\widehat{\varphi_1}(t),\widehat{\varphi_2}(t),\widehat{I_2}(t)) 
\cdot
\partial_{I_{1}} \mathcal{K}_{\varepsilon} (\widehat{\varphi_1}(t),\widehat{\varphi_2}(t),\widehat{I_1}(t),\widehat{I_2}(t))
\\
&= \dfrac{\partial_{I_{2}} \mathcal{K}_{\varepsilon} (\varepsilon,\widehat{\varphi_1}(t),\widehat{\varphi_2}(t),\widehat{I_2}(t))}{\partial_{I_{1}} \mathcal{K}_{\varepsilon} (\widehat{\varphi_1}(t),\widehat{\varphi_2}(t),\widehat{I_1}(t),\widehat{I_2}(t))}
\cdot
\partial_{I_{1}} \mathcal{K}_{\varepsilon} (\widehat{\varphi_1}(t),\widehat{\varphi_2}(t),\widehat{I_1}(t),\widehat{I_2}(t))
\\
&= \partial_{I_{2}} \mathcal{K}_{\varepsilon} (\widehat{\varphi_1}(t),\widehat{\varphi_2}(t),\widehat{I_1}(t),\widehat{I_2}(t))
\end{align*}
and, analogously, 
\begin{align*}
\dot{\widehat{I_2}}(t) 
&= \widetilde{I_2}'(\widetilde{\varphi_1}(t)) \dot{\widetilde{\varphi_1}}(t)
= \partial_{\varphi_{2}} \mathcal{I}(\varepsilon,\widetilde{\varphi_1}(t),\widetilde{\varphi_2}(\widetilde{\varphi_1}(t)),\widetilde{I_2}(\widetilde{\varphi_1}(t))) 
\cdot
\dot{\widetilde{\varphi_1}}(t)
\\
&= \partial_{\varphi_{2}} \mathcal{I}(\varepsilon,\widehat{\varphi_1}(t),\widehat{\varphi_2}(t),\widehat{I_2}(t)) 
\cdot
\partial_{I_{1}} \mathcal{K}_{\varepsilon} (\widehat{\varphi_1}(t),\widehat{\varphi_2}(t),\widehat{I_1}(t),\widehat{I_2}(t))
\\
&= -\partial_{\varphi_{2}} \mathcal{K}_{\varepsilon} (\widehat{\varphi_1}(t),\widehat{\varphi_2}(t),\widehat{I_1}(t),\widehat{I_2}(t)).
\end{align*}
At last, from the above three formulas and the implicit function theorem, we deduce that
\begin{align*}
\dot{\widehat{I_1}}(t)  
&= \partial_{\varphi_1} \mathcal{I}(\varepsilon,\widetilde{\varphi_1}(t),\widetilde{\varphi_2}(\widetilde{\varphi_1}(t)),\widetilde{I_2}(\widetilde{\varphi_1}(t))) \cdot \dot{\widetilde{\varphi_1}}(t) 
\\
&\quad + \partial_{\varphi_2} \mathcal{I}(\varepsilon,\widetilde{\varphi_1}(t),\widetilde{\varphi_2}(\widetilde{\varphi_1}(t)),\widetilde{I_2}(\widetilde{\varphi_1}(t))) \cdot
\widetilde{\varphi_2}'(\widetilde{\varphi_1}(t)) \dot{\widetilde{\varphi_1}}(t)
\\
&\quad + \partial_{I_2} \mathcal{I}(\varepsilon,\widetilde{\varphi_1}(t),\widetilde{\varphi_2}(\widetilde{\varphi_1}(t)),\widetilde{I_2}(\widetilde{\varphi_1}(t))) \cdot
\widetilde{I_2}'(\widetilde{\varphi_1}(t)) \dot{\widetilde{\varphi_1}}(t)
\\
&= -\partial_{\varphi_{1}} \mathcal{K}_{\varepsilon} (\widehat{\varphi_1}(t),\widehat{\varphi_2}(t),\widehat{I_1}(t),\widehat{I_2}(t)),
\end{align*}
for all $t$. We have thus proved that $(\widehat{\varphi_1},\widehat{\varphi_2},\widehat{I_1},\widehat{I_2})$ is a solution of system \eqref{eq-sispertazioneangolo}, as claimed; 
moreover, its energy is
\begin{equation*}
\mathcal{K}_{\varepsilon} (\widehat{\varphi_1}(t),\widehat{\varphi_2}(t),\widehat{I_1}(t),\widehat{I_2}(t))
= \mathcal{K}_{\varepsilon} (\widetilde{\varphi_1}(t),\widetilde{\varphi_2}(\widetilde{\varphi_1}(t)),\widetilde{\mathcal{I}_{\varepsilon}}(\widetilde{\varphi_1}(t)),\widetilde{I_2}(\widetilde{\varphi_1}(t))) = K^{*},
\end{equation*}
by the implicit function theorem.

We now prove that $(\widehat{\varphi_1},\widehat{\varphi_2},\widehat{I_1},\widehat{I_2})$ is $|T_{\varepsilon}|$-periodic and have $\mathrm{sgn}(\partial_{I_1}\mathcal{K}_0(I^*)) (m_1,m_2)$ as winding vector on the period $|T_{\varepsilon}|$. By hypothesis we know that
\begin{equation}\label{eq-per-proof}
\bigl{(} \widetilde{\varphi_2}(2\pi m_{1}), \widetilde{I_2}(2\pi m_{1}) \bigr{)}
= \bigl{(} \widetilde{\varphi_2}(0) + 2\pi m_{2}, \widetilde{I_2}(0) \bigr{)}
\end{equation}
Recalling that $T_{\varepsilon} ={t}_{\varepsilon}(2\pi m_1)$ and so $\widetilde{\varphi_1}(T_{\varepsilon})=2\pi m_{1}$, we thus find that
\begin{align*}
\bigl{(}\widehat{\varphi_1}(T_{\varepsilon}),\widehat{\varphi_2}(T_{\varepsilon}),\widehat{I_1}(T_{\varepsilon}),\widehat{I_2}(T_{\varepsilon})\bigr{)}
&= 
\bigl{(}\widetilde{\varphi_1}(T_{\varepsilon}),\widetilde{\varphi_2}(\widetilde{\varphi_1}(T_{\varepsilon})),\widetilde{\mathcal{I}_{\varepsilon}}(\widetilde{\varphi_1}(T_{\varepsilon})),\widetilde{I_2}(\widetilde{\varphi_1}(T_{\varepsilon}))\bigr{)}
\\
&= 
\bigl{(}2\pi m_{1},\widetilde{\varphi_2}(2\pi m_{1}),\widetilde{\mathcal{I}_{\varepsilon}}(2\pi m_{1}),\widetilde{I_2}(2\pi m_{1})\bigr{)}
\\
&= 
\bigl{(}2\pi m_{1},\widetilde{\varphi_2}(0)+2\pi m_{2},\widetilde{\mathcal{I}_{\varepsilon}}(0),\widetilde{I_2}(0)\bigr{)},
\end{align*}
where in the last equality we have exploited \eqref{eq-per-proof} and the fact that
\begin{align*}
\widetilde{\mathcal{I}_{\varepsilon}}(2\pi m_{1}) 
&=
\mathcal{I}(\varepsilon,2\pi m_{1}, \widetilde{\varphi_2} (2\pi m_{1}),\widetilde{I_2}(2\pi m_{1}))
\\
&=
\mathcal{I}(\varepsilon,2\pi m_{1}, \widetilde{\varphi_2} (0) + 2\pi m_{2},\widetilde{I_2}(0))
=
\widetilde{\mathcal{I}_{\varepsilon}}(0),
\end{align*}
since $\mathcal{I}$ is $2\pi$-periodic in each of the variables $\varphi_1$ and $\varphi_2$.
Recalling that 
\begin{equation*}
\bigl{(}\widehat{\varphi_1}(0),\widehat{\varphi_2}(0),\widehat{I_1}(0),\widehat{I_2}(0)\bigr{)}= \bigl{(}0,\widetilde{\varphi_2}(0),\widetilde{\mathcal{I}_{\varepsilon}}(0),\widetilde{I_2}(0)\bigr{)},
\end{equation*}
we thus have
\begin{equation*}
\bigl{(}\widehat{\varphi_1}(T_{\varepsilon}),\widehat{\varphi_2}(T_{\varepsilon}),\widehat{I_1}(T_{\varepsilon}),\widehat{I_2}(T_{\varepsilon})\bigr{)} = \bigl{(}\widehat{\varphi_1}(0)+2\pi m_1,\widehat{\varphi_2}(0)+2\pi m_2,\widehat{I_1}(0),\widehat{I_2}(0)\bigr{)}.
\end{equation*}
Taking into account that $\mathrm{sgn}(T_\varepsilon) = \mathrm{sgn}(\partial_{I_1}\mathcal{K}_0(I^*)) (m_1,m_2)$, this means that the solution
$(\widehat{\varphi_1},\widehat{\varphi_2},\widehat{I_1},\widehat{I_2})$ is $|T_{\varepsilon}|$-periodic with $\mathrm{sgn}(\partial_{I_1}\mathcal{K}_0(I^*)) (m_1,m_2)$ as winding vector on the period $|T_{\varepsilon}|$. Since the integers $m_1$ and $m_2$ are coprime, 
$|T_{\varepsilon}|$ is thus the minimal period.

Finally, let $(\widetilde{\varphi_2}^1,\widetilde{I_2}^1)$ and $(\widetilde{\varphi_2}^2,\widetilde{I_2}^2)$ be distinct solutions of \eqref{eq-sisridotto} in the sense of \eqref{formula-distinti3}. We prove that the corresponding solutions $(\widehat{\varphi}^1,\widehat{I}^1)$ and $(\widehat{\varphi}^2,\widehat{I}^2)$ of \eqref{eq-sispertazioneangolo} are distinct in the sense of \eqref{formula-distinti2}.
By contradiction, we assume that there exist $\tau\in\mathopen{[}0,T_{\varepsilon}\mathclose{[}$ and $k_{1},k_{2}\in\mathbb{Z}$ such that
\begin{align*}
&\bigl{(} \widehat{\varphi_{1}}^1(t+\tau),\widehat{\varphi_{2}}^1(t+\tau),\widehat{I_{1}}^1(t+\tau),\widehat{I_{2}}^1(t+\tau) \bigr{)}=
\\
&=
\bigl{(} \widehat{\varphi_{1}}^2(t)+2\pi k_{1},\widehat{\varphi_{2}}^2(t)+2\pi k_{2},\widehat{I_{1}}^2(t),\widehat{I_{2}}^2(t)\bigr{)}, \quad \text{for all $t\in\mathbb{R}$.}
\end{align*}
Therefore, recalling \eqref{eq-lem}, it is easy to see that the functions $(\widetilde{\varphi_{2}}^1(\cdot),\widetilde{I_{2}}^1(\cdot))$ and $(\widetilde{\varphi_{2}}^2(\cdot-2\pi k_{1})+2\pi k_{2},\widetilde{I_{2}}^2(\cdot-2\pi k_{1}))$ satisfy the same initial condition at $t=2\pi k_1$. Since they are both solutions of \eqref{eq-sisridotto}, they must coincide for all $t \in \mathbb{R}$ and thus they are not distinct solutions of \eqref{eq-sisridotto} in the sense of \eqref{formula-distinti3}, a contradiction.
\end{proof}
	
In order to find $2\pi m_1$-periodic solutions of the Hamiltonian system \eqref{eq-sisridotto}, we are going to apply Theorem~\ref{KAMePB} to its Poincar\'e map at time $T = 2\pi$, that is, to the map
\begin{equation}\label{defmapM}
M_{\varepsilon}(\varphi_2,I_2) = (\widetilde{\varphi_2}(2\pi;\varepsilon,\varphi_2,I_2),\widetilde{I_2}(2\pi;\varepsilon,\varphi_2,I_2)), 
\end{equation}
where the notation $(\widetilde{\varphi_2}(\cdot;\varepsilon,\varphi_2,I_2),\widetilde{I_2}(\cdot;\varepsilon,\varphi_2,I_2))$
stands for the unique solution of system \eqref{eq-sisridotto} satisfying the initial condition
$(\widetilde{\varphi_2}(0),\widetilde{I_2}(0))= (\varphi_2,I_2)$.  
As well known, a $(m_1,m_2)$-periodic point of $M_{\varepsilon}$ corresponds to (the initial conditions of) a $2\pi m_1$-periodic solution of
\eqref{eq-sisridotto} with winding number $m_2$; moreover, two periodic points are distinct in the sense of \eqref{formula-distinti1} if and only if the corresponding periodic solutions are distinct in the sense of \eqref{formula-distinti3}.

For $\varepsilon = 0$, the map $M_{\varepsilon}$ has an explicit expression. Indeed, recalling \eqref{eq-I0} we find that
\begin{equation*}
(\widetilde{\varphi_2}(t;0,\varphi_2,I_2),\widetilde{I_2}(t;0,\varepsilon,\varphi_2,I_2)) = (\varphi_2- t \,\mathcal{I}_0'(I_2),I_2)
\end{equation*}
and so
\begin{equation}\label{defM0}
M_0(\varphi_2,I_2) = (\varphi_2 + \alpha(I_2),I_2),
\end{equation}
where
\begin{equation*}
\alpha(I_2) = -  2\pi\,\mathcal{I}_0'(I_2).
\end{equation*}
Let us notice that the expression is valid as soon as the function $\mathcal{I}_0$ is well-defined, that is, 
for $(\varphi_2,I_2) \in \mathbb{R} \times [a,b]$. By standard continuous dependence argument (taking into account the $2\pi$-periodicity with respect to $\varphi_1$ of the function $\mathcal{K}^{\rm red}_{\varepsilon}$), we thus have that, up to shrinking $\bar{\varepsilon}$ if necessary, the map $M_{\varepsilon}$ is well-defined on the strip $\mathbb{R} \times [a,b]$ for $\vert \varepsilon \vert \leq \bar{\varepsilon}$.

As first step for the application of Theorem~\ref{KAMePB}, we now prove the following result.
	
\begin{lemma}\label{lem1}
For every $\varepsilon \in \mathbb{R}$ with $\vert \varepsilon \vert \leq \bar{\varepsilon}$, the map $M_{\varepsilon} \colon \mathbb{R} \times [a,b] \to \mathbb{R}^2$ is exact-symplectic, of class $\mathcal{C}^4$ and of the form
\begin{equation*}
M_{\varepsilon}(\varphi_2,I_2) = (\varphi_2 + \alpha(I_2) + \mu_1(\varphi_2,I_2;\varepsilon), I_2 + \mu_1(\varphi_2,I_2;\varepsilon)),
\end{equation*}
where $\mu_1,\mu_2$ are $2\pi$-periodic in $\theta$ and satisfy $\mu_1(\varphi_2,I_2;0) = \mu_1(\varphi_2,I_2;0) = 0$ as well as  
\begin{equation}\label{stimamu}
\Vert \mu_1(\cdot,\cdot; \varepsilon) \Vert_{\mathcal{C}^4} + \Vert \mu_2(\cdot,\cdot;\varepsilon) \Vert_{\mathcal{C}^4} \to 0, \quad \text{as $\varepsilon \to 0$.}
\end{equation}
\end{lemma}
	
\begin{proof}
The fact that $M_{\varepsilon}$ is exact-symplectic is a well-known consequence of the fact that the Hamiltonian function \eqref{eq-defkridotta} of system \eqref{eq-sisridotto} is periodic in the variable $\varphi_2$ (see for instance \cite[Proposition~2.4]{BoOr-14}).
This also implies that $\mu_1$ and $\mu_2$ are $2\pi$-periodic in $\theta$; moreover, \eqref{defM0} yields $\mu_1(\varphi_2,I_2,0) = \mu_1(\varphi_2,I_2,0) = 0$. Thus, it remains to shows that $M_{\varepsilon}$ is of class $\mathcal{C}^4$ and that \eqref{stimamu} holds true.
	
To this end, we observe that, since the Hamiltonian is of class $\mathcal{C}^6$ (and, thus, the Hamiltonian vector field of class $\mathcal{C}^5$), from the standard theory of ordinary differential equations we know that the map
\begin{equation*}
[-\bar{\varepsilon},\bar{\varepsilon}] \times \mathbb{R} \times [a,b] \ni (\varepsilon,\varphi_2,I_2) \mapsto \widehat M(\varepsilon,\varphi_2,I_2):= M_{\varepsilon}(\varphi_2,I_2) 
\end{equation*}
is of class $\mathcal{C}^5$. This implies on one hand that, for every fixed $\varepsilon$, the map $M_{\varepsilon}$ is of class 
$\mathcal{C}^5$ (and, thus, of class $\mathcal{C}^4$ as required). On the other hand, it guarantees that the map
$\frac{\partial \widehat M}{\partial \varepsilon}$ is of class $\mathcal{C}^4$, and thus there is a constant $C > 0$ such that
\begin{equation*}
\left\Vert \frac{\partial \widehat M}{\partial \varepsilon}(\varepsilon,\cdot,\cdot) \right\Vert_{\mathcal{C}^4([0,2\pi] \times [a,b])} \leq C, \quad \text{for every $\vert \varepsilon \vert \leq \bar{\varepsilon}$.}
\end{equation*}
Since, for every $(\varphi_2,I_2) \in [0,2\pi] \times [a,b]$, it holds that
\begin{equation*}
(\mu_1(\varphi_2,I_2;\varepsilon),\mu_2(\varphi_2,I_2;\varepsilon)) = \widehat M(\varepsilon,\varphi_2,I_2) - \widehat M(0,\varphi_2,I_2)
=  \varepsilon \int_0^1 \frac{\partial \widehat M}{\partial \varepsilon}(\varepsilon s,\varphi_2,I_2)\,\mathrm{d}s,
\end{equation*}
condition \eqref{stimamu} directly follows (cf.~\cite[Proposition~6.4]{Or-01} for a similar argument).
\end{proof}
	
Now we turn to the proof of the validity of the conditions \eqref{condizione1} and \eqref{condizione2} of Theorem~\ref{KAMePB}.
	
\begin{lemma}\label{lem2}
It holds that
\begin{equation*}
\alpha(I_2^*) = 2\pi \frac{n_2}{n_1} \quad \text{ and } \quad \alpha'(I_2^*) \neq 0.
\end{equation*}
\end{lemma}
					
\begin{proof}
From \eqref{eq-derfunzimpl} for $\varepsilon = 0$ we immediately obtain
\begin{equation}\label{eq-prima}
\mathcal{I}_0'(I_2) = - \frac{\partial_{I_2}\mathcal{K}_0(\mathcal{I}_0(I_2),I_2)}{\partial_{I_1}\mathcal{K}_0(\mathcal{I}_0(I_2),I_2)}.
\end{equation}
Therefore, recalling \eqref{eq-I0*} and \eqref{eq-ipoazioneangolo}, we find
\begin{equation*}
\alpha(I_2^*) = - 2\pi \mathcal{I}_0'(I_2^*) = 2\pi \frac{\partial_{I_2}
\mathcal{K}_0(I^*)}{\partial_{I_1}\mathcal{K}_0(I^*)} = 2\pi \frac{n_2}{n_1}.
\end{equation*}
On the other hand, differentiating \eqref{eq-prima} we obtain
\begin{align*}
\mathcal{I}_0''(I_2) & = - 
\dfrac{\left( \partial^2_{I_1 I_2}\mathcal{K}_0(\mathcal{I}_0(I_2),I_2) \mathcal{I}_0'(I_2) + \partial^2_{I_2 I_2}\mathcal{K}_0(\mathcal{I}_0(I_2),I_2) \right)}{(\partial_{I_1}\mathcal{K}_0(\mathcal{I}_0(I_2),I_2))^2} \, \partial_{I_1}\mathcal{K}_0(\mathcal{I}_0(I_2),I_2)
\\
& \quad + \dfrac{\left( \partial^2_{I_1 I_1}\mathcal{K}_0(\mathcal{I}_0(I_2),I_2) \mathcal{I}_0'(I_2) + \partial^2_{I_1 I_2}\mathcal{K}_0(\mathcal{I}_0(I_2),I_2) \right)}{(\partial_{I_1}\mathcal{K}_0(\mathcal{I}_0(I_2),I_2))^2} \, \partial_{I_2}\mathcal{K}_0(\mathcal{I}_0(I_2),I_2). 
\end{align*}
Plugging \eqref{eq-prima} in the above formula,  we have
\begin{align*}
\mathcal{I}_0''(I_2) &= -\frac{1}{(\partial_{I_1}\mathcal{K}_0)^3} \biggl{[} \partial^2_{I_1 I_1}\mathcal{K}_0 \cdot (\partial_{I_2}\mathcal{K}_{0})^2 - 2\partial^2_{I_1 I_2}\mathcal{K}_0 \cdot \partial_{I_1}\mathcal{K}_0  \cdot \partial_{I_2}\mathcal{K}_0 +\partial^2_{I_2 I_2}\mathcal{K}_0 \cdot (\partial_{I_1}\mathcal{K}_0)^2 \biggr{]}
\\
& = -\frac{1}{(\partial_{I_1}\mathcal{K}_0)^3} \langle \nabla^2 \mathcal{K}_0 (-\partial_{I_2}\mathcal{K}_0, \partial_{I_1}\mathcal{K}_0), (-\partial_{I_2}\mathcal{K}_0, \partial_{I_1}\mathcal{K}_0) \rangle,
\end{align*}
where we agree, for simplicity of notation, that the right-hand side is evaluated at $(\mathcal{I}_0(I_2),I_2)$.
For $I_2 =I_2^*$, recalling \eqref{eq-I0*} and \eqref{eq-iponondegenere} we thus find
\begin{align*}
\alpha'(I_2^*) &= \frac{2\pi}{(\partial_{I_1}\mathcal{K}_0(I^*))^3} \langle \nabla^2 \mathcal{K}_0(I^*) (-\partial_{I_2}\mathcal{K}_0(I^*), \partial_{I_1}\mathcal{K}_0(I^*)), (-\partial_{I_2}\mathcal{K}_0(I^*), \partial_{I_1}\mathcal{K}_0(I^*)) \rangle 
\\
&\neq 0,
\end{align*}
thus concluding the proof.
\end{proof}

\begin{remark}\label{rem-alternativa}
Let us point out that from the proof of Lemma~\ref{lem2} it follows that the non-degeneracy condition \eqref{eq-iponondegenere} is equivalent to the fact that the function
\begin{equation*}
I_2 \mapsto \frac{\partial_{I_2}\mathcal{K}_0(\mathcal{I}_0(I_2),I_2)}{\partial_{I_1}\mathcal{K}_0(\mathcal{I}_0(I_2),I_2)}
\end{equation*}
has non-zero derivative at $I_2^*$ (and, thus, is a local diffeomorphism). We will use this observation later on.
\hfill$\lhd$
\end{remark}

We are now in a position to summarize the whole discussion so as to give the proof of Theorem~\ref{teo-pbapplicato}.

\begin{proof}[Proof of Theorem~\ref{teo-pbapplicato}]
In view of Lemma~\ref{lem1} and Lemma~\ref{lem2}, we can apply Theorem~\ref{KAMePB} to the map $M_{\varepsilon}$
defined in \eqref{defmapM}. Hence, we find $\varepsilon^* > 0$ and $\delta^* > 0$ such that, for every pair of coprime integers $(m_1,m_2) \in (\mathbb{N}\setminus\{0\}) \times \mathbb{Z}$ satisfying
\begin{equation*}
\left\vert \frac{m_2}{m_1} - \frac{n_2}{n_1}\right \vert < \delta^*,
\end{equation*}
and for every $\varepsilon$ satisfying $\vert\varepsilon \vert < \varepsilon^*$, the map $M_{\varepsilon}$ has at least two
distinct $(m_1,m_2)$-periodic points. These two distinct periodic points correspond to two distinct $2\pi m_1$-periodic solutions of
\eqref{eq-sisridotto} with winding number $m_2$. Hence, by Lemma~\ref{lem-solridotto}, we find two distinct periodic solutions of \eqref{eq-sispertazioneangolo}, having energy $K^*$ and winding vector 
\begin{equation*}
(l_1,l_2) = \mathrm{sgn}(\partial_{I_1}\mathcal{K}_0(I^*)) (m_1,m_2).
\end{equation*}
Thus, Theorem~\ref{teo-pbapplicato} is proved.
\end{proof}		
	
\begin{remark}\label{rem2}
According to the observation in Remark~\ref{rem-alternativa}, for any pair of coprime integers $(m_1,m_2) \in (\mathbb{N}\setminus\{0\}) \times \mathbb{Z}$ satisfying \eqref{hpm1} there exists a unique value $I_2$ in a neighborhood of $I_2^*$ such that
\begin{equation*}
\frac{\partial_{I_2}\mathcal{K}_0(\mathcal{I}_0(I_2),I_2)}{\partial_{I_1}\mathcal{K}_0(\mathcal{I}_0(I_2),I_2)} = \frac{m_2}{m_1}.
\end{equation*}
It is easily proved that, for $\varepsilon \to 0$, the two periodic solutions with winding vector $(l_1,l_2) = \mathrm{sgn}(\partial_{I_1}\mathcal{K}_0(I^*)) (m_1,m_2)$ converge to the invariant torus $(\mathcal{I}_0(I_2),I_2)$ of the unperturbed problem; moreover, their minimal periods converge to $\vert 2\pi l_1/\partial_{I_1}\mathcal{K}_0(I^*) \vert$. Notice, in particular, that the minimal periods of these solutions can be arbitrarily large. 
\hfill$\lhd$
\end{remark}

\section{The main result: bifurcation of closed orbits via time maps}\label{section-main}

In this section we investigate bifurcation of closed orbits for a specific class of Hamiltonian systems with two degrees of freedom;
by passing to action-angle variables, such systems will be converted into systems of the type \eqref{eq-sispertazioneangolo} in order to make possible the use of the abstract result of Section~\ref{section-2}.

More precisely, let us consider a planar motion identified by a position $x\in \mathbb{R}^2\setminus \{0\}$, written in polar coordinates as $x=re^{i\vartheta}$, with $r>0$. We deal with perturbed Hamiltonians of the type
\begin{equation}\label{eq-hamper}
\mathcal{H}_{\varepsilon}(r,\vartheta,p_r,p_\vartheta) = \mathcal{H}_0(r,\vartheta,p_r,p_\vartheta) + \varepsilon \widetilde{\mathcal{H}} (\varepsilon,r,\vartheta,p_r,p_\vartheta),
\end{equation}
where $\varepsilon$ is a small real parameter, $\widetilde{\mathcal{H}}\colon (-\hat{\varepsilon},\hat{\varepsilon}) \times \Xi \times \mathbb{R} \times \mathbb{R}^2 \to \mathbb{R}$, with $\Xi \subset (0,+\infty)$ an interval, is a function which is $2\pi$-periodic in the third variable, and $\mathcal{H}_0\colon \Xi \times \mathbb{R} \times \mathbb{R}^2$ is a function of the form 
\begin{equation} \label{eq-hamunpertastratta}
\mathcal{H}_0(r,\vartheta,p_r,p_\vartheta)=\displaystyle \dfrac{1}{2}\, \left(\alpha(r)\, p_{r}^2+\dfrac{1}{r^2} \, p_{\vartheta}^2\right) - V(r),
\quad (r,\vartheta,p_r,p_\vartheta)\in \Xi \times \mathbb{R} \times \mathbb{R}^2,
\end{equation}
for suitable functions $\alpha, V \colon \Xi \to \mathbb{R}$, with $\alpha (r) > 0$, for every $r\in \Xi$. In order to simplify the exposition, and since this is not restrictive for the applications that we are going to present, we assume that all the functions are of class $\mathcal{C}^\infty$ in their domain of definition.
	
The choice of a Hamiltonian $\mathcal{H}_0$ of the form \eqref{eq-hamunpertastratta} is motivated by the applications which we are going to discuss in the sequel. More precisely, let us notice that, on one hand, all central force problems in the plane, that is $\ddot x = V'(\vert x \vert) x/\vert x \vert$, can be written in the form \eqref{eq-hamunpertastratta}, cf.~\eqref{cenfor}. Moreover, the formulation \eqref{eq-hamunpertastratta} covers also the case of the reduced Hamiltonian coming from the study of equatorial geodesic motion in the Schwarzschild space time, see Section~\ref{sub-hamschw}.
	
In the following, $T$-periodic solutions of the Hamiltonian system associated with \eqref{eq-hamper} are meant as solutions $(r,\vartheta,p_r,p_\vartheta)$ such that $r,p_r,p_\vartheta$ are $T$-periodic and
\begin{equation*}
\vartheta(T) -\vartheta(0) \in 2\pi \mathbb{Z}. 
\end{equation*}
Accordingly, the planar motion $x = r e^{i\vartheta}$ is $T$-periodic in the usual sense. Moreover, taking into account both the invariance by time-translation as well as the $2\pi$-periodicity in $\vartheta$ of \eqref{eq-hamper}, we say that two $T$-periodic solutions $(r^1,\vartheta^1,p_r^1,p_\vartheta^1)$ and $(r^2,\vartheta^2,p_r^2,p_\vartheta^2)$ of \eqref{eq-hamper} are distinct if 
\begin{equation*}
(r^1(t+\tau),\vartheta^1(t+\tau),p_r^1(t+\tau),p_\vartheta^1(t+\tau))- (r^2(t),\vartheta^2(t),p_r^2(t),p_\vartheta^2(t)) \notin 
\{0\} \times 2\pi \mathbb{Z} \times \{(0,0)\}, 
\end{equation*}
for every $t,\tau \in \mathbb{R}$.
This implies that the corresponding cartesian solutions $x^1 = r^1 e^{i\vartheta^i}$ and $x^2 = r^2 e^{i\vartheta^2}$
satisfy $x^1(\cdot+\tau) \not\equiv x^2(\cdot)$ for every $\tau \in \mathbb{R}$.

\subsection{Preliminaries for the unperturbed problem} \label{sec-prelastrattononpert}
	
Let us observe that for the Hamiltonian system associated with \eqref{eq-hamunpertastratta} the radial motion is completely governed by an \textit{effective potential} independent from the angular coordinate $\vartheta$. Indeed, let us first notice that the Hamiltonian system is
\begin{equation} \label{eq-sisthamunpert}
\begin{cases}
\, \dot{r} = \alpha(r)\, p_r, 
& \; \dot{p_r} = -\dfrac{1}{2}\, \alpha'(r)\, p_r^2 +\dfrac{1}{r^3}\, p_\vartheta^2 + V'(r), 
\\
\, \dot{\vartheta} = \dfrac{1}{r^2}\, p_{\vartheta},
& \;
\dot{p_\vartheta} =  0,
\end{cases}
\end{equation}
thus implying that the \textit{angular momentum} $p_\vartheta$ is constant. 
	
From now on, we deal with the case $p_\vartheta \neq 0$, namely, the solution $x = r e^{i\vartheta}$ is non-rectilinear. We also notice that the time inversion $t \mapsto -t$ yields a change of sign for $p_\vartheta$. Hence, in what follows we assume that $p_\vartheta$ is positive and we denote by $L > 0$ its value. Then, the conservation of the Hamiltonian (the \textit{energy}) reads as
\begin{equation*}
\dfrac{1}{2}\, \left(\alpha(r)\, p_{r}^2+\dfrac{L^2}{r^2} \right) - V(r)=H,
\end{equation*}
for some $H\in \mathbb{R}$. From this relation, taking into account the first equation in \eqref{eq-sisthamunpert}, we easily obtain
\begin{equation} \label{eq-consenergia2}
\dfrac{1}{2}\, \dot{r}^2+\left(\dfrac{1}{2}\, \dfrac{L^2}{r^2}-V(r)-H\right)\, \alpha(r)=0.
\end{equation}
For every $(H,L)\in \mathbb{R}\times (0,+\infty)$, let us define the effective potential $Z(\cdot;H, L)\colon \Xi\to \mathbb{R}$ by
\begin{equation} \label{eq-defW}
Z(r; H, L)=\left(\dfrac{1}{2}\, \dfrac{L^2}{r^2}-V(r)-H\right)\, \alpha(r).
\end{equation}
Hence, we can write \eqref{eq-consenergia2} as
\begin{equation} \label{eq-consenergia3}
\dfrac{1}{2}\, \dot{r}^2+Z(r; H, L)=0.
\end{equation}
	
In order to ensure the existence of periodic orbits of \eqref{eq-sisthamunpert} and to be able to construct action-angle variables associated with the Hamiltonian \eqref{eq-hamunpertastratta}, we need some assumptions on $Z$. More precisely, we suppose that
\begin{enumerate}[leftmargin=26pt,labelsep=8pt,label=\textup{$(\star)$}]
\item there exist an open set $\Lambda \subset \mathbb{R}^2$ and three continuous functions $r_0, r_\pm\colon \Lambda \to \Xi$, with $r_-<r_0<r_+$, such that
\begin{equation} \label{eq-ipo1}
Z(r_0(H,L); H, L)<0,\quad Z'(r_0(H,L); H, L)=0,\quad Z''(r_0(H,L); H, L)>0
\end{equation}
and
\begin{equation} \label{eq-ipo2}
\begin{aligned}
&Z(r_\pm (H,L); H, L)=0,
\\
&Z'(r; H, L)(r-r_0(H,L))>0, \; \text{for every $r\in [r_-(H,L),r_+(H,L)]\setminus\{r_{0}(H,L)\}$,}
\end{aligned}
\end{equation}
for every $(H,L)\in \Lambda$.
\label{hp-star}
\end{enumerate}

See Figure~\ref{fig-01} for a graphical representation.

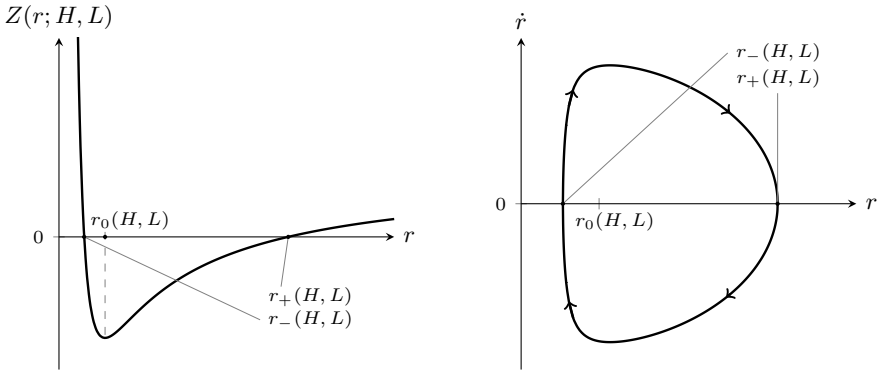
\begin{figure}[htb]
\centering
\begin{tikzpicture}
\begin{axis}[
  tick label style={font=\scriptsize},
  axis y line=left, 
  axis x line=middle,
  xtick={0},
  ytick={0},
  xticklabels={},
  yticklabels={$0$},
  xlabel={\small $r$},
  ylabel={\small $Z(r;H,L)$},
every axis x label/.style={
    at={(ticklabel* cs:1.0)},
    anchor=west,
},
every axis y label/.style={
    at={(ticklabel* cs:1.0)},
    anchor=south,
},
  width=6cm,
  height=6cm,
  xmin=0,
  xmax=8,
  ymin=-0.4,
  ymax=0.6]
\addplot [color=black,line width=0.9pt,smooth] coordinates {(0.1, 42.2567) (0.179, 11.1014) (0.258, 4.3267) (0.337, 1.94107) (0.416, 0.887245) (0.495, 0.357586) (0.574, 0.0694774) (0.653, -0.094917) (0.732, -0.191016) (0.811, -0.247283) (0.89, -0.279344) (0.969, -0.296244) (1.048, -0.303435) (1.127, -0.304323) (1.206, -0.301089) (1.285, -0.295166) (1.364, -0.287507) (1.443, -0.278761) (1.522, -0.269371) (1.601, -0.25964) (1.68, -0.24978) (1.759, -0.239936) (1.838, -0.230209) (1.917, -0.220666) (1.996, -0.211352) (2.075, -0.202295) (2.154, -0.193513) (2.233, -0.185014) (2.312, -0.1768) (2.391, -0.168869) (2.47, -0.161217) (2.549, -0.153838) (2.628, -0.146722) (2.707, -0.139861) (2.786, -0.133247) (2.865, -0.126868) (2.944, -0.120715) (3.023, -0.114779) (3.102, -0.109051) (3.181, -0.10352) (3.26, -0.0981792) (3.339, -0.0930192) (3.418, -0.088032) (3.497, -0.0832098) (3.576, -0.0785453) (3.655, -0.0740315) (3.734, -0.0696615) (3.813, -0.0654292) (3.892, -0.0613283) (3.971, -0.0573533) (4.05, -0.0534986) (4.129, -0.0497591) (4.208, -0.04613) (4.287, -0.0426065) (4.366, -0.0391842) (4.445, -0.0358591) (4.524, -0.0326271) (4.603, -0.0294844) (4.682, -0.0264275) (4.761, -0.023453) (4.84, -0.0205577) (4.919, -0.0177384) (4.998, -0.0149924) (5.077, -0.0123167) (5.156, -0.00970886) (5.235, -0.00716624) (5.314, -0.00468648) (5.393, -0.00226729) (5.472, 0.0000934928) (5.551, 0.00239795) (5.63, 0.00464806) (5.709, 0.00684572) (5.788, 0.00899272) (5.867, 0.0110908) (5.946, 0.0131416) (6.025, 0.0151467) (6.104, 0.0171076) (6.183, 0.0190258) (6.262, 0.0209026) (6.341, 0.0227393) (6.42, 0.0245373) (6.499, 0.0262977) (6.578, 0.0280218) (6.657, 0.0297105) (6.736, 0.0313651) (6.815, 0.0329865) (6.894, 0.0345757) (6.973, 0.0361337) (7.052, 0.0376614) (7.131, 0.0391597) (7.21, 0.0406294) (7.289, 0.0420713) (7.368, 0.0434863) (7.447, 0.044875) (7.526, 0.0462382) (7.605, 0.0475766) (7.684, 0.0488909) (7.763, 0.0501817) (7.842, 0.0514497) (7.921, 0.0526954) (8., 0.0539194)};
\draw [color=gray, dashed, line width=0.3pt] (axis cs:1.1,0.015)--(axis cs:1.1,-0.304574);
\node at (axis cs: 1.7,0.05) {\scriptsize{$r_{0}(H,L)$}};
\draw [color=gray, line width=0.3pt] (axis cs: 0.6, 0)--(axis cs:4.8,-0.25);
\draw [color=gray, line width=0.3pt] (axis cs: 5.472, 0)--(axis cs:5.3,-0.15);
\node at (axis cs: 6,-0.18) {\scriptsize{$r_{+}(H,L)$}};
\node at (axis cs: 6,-0.25) {\scriptsize{$r_{-}(H,L)$}};
\fill (axis cs: 0.6,0) circle (0.8pt);
\fill (axis cs: 1.1,0) circle (0.8pt);
\fill (axis cs: 5.472,0) circle (0.8pt);
\end{axis}
\end{tikzpicture}
\quad\quad
\begin{tikzpicture}
\begin{axis}[
  tick label style={font=\scriptsize},
  axis y line=left, 
  axis x line=middle,
  xtick={1.2},
  ytick={0},
  xticklabels={},
  yticklabels={$0$},
  xlabel={\small $r$},
  ylabel={\small $\dot r$},
every axis x label/.style={
    at={(ticklabel* cs:1.0)},
    anchor=west,
},
every axis y label/.style={
    at={(ticklabel* cs:1.0)},
    anchor=south,
},
  width=6cm,
  height=6cm,
  xmin=0.2,
  xmax=4.5,
  ymin=-3.3,
  ymax=3.3]
\draw[->] [color=black,line width=0.9pt] (axis cs: 0.812, 1.94068)--(axis cs: 0.868, 2.27431);
\draw[->] [color=black,line width=0.9pt] (axis cs: 2.828, 1.86766)--(axis cs: 2.884, 1.79682);
\draw[->] [color=black,line width=0.9pt] (axis cs: 0.868, -2.27431)--(axis cs: 0.812, -1.94068);
\draw[->] [color=black,line width=0.9pt] (axis cs: 2.884, -1.79682)--(axis cs: 2.828, -1.86766);
\addplot [color=black,line width=0.9pt,smooth] coordinates {(0.7355, 0) (0.7368, 0.16485) (0.7384, 0.381708) (0.74, 0.511707) (0.7416, 0.612917) (0.7432, 0.698019) (0.7448, 0.772395) (0.7464, 0.838936) (0.748, 0.899413) (0.7496, 0.955007) (0.7512, 1.00655) (0.7528, 1.05467) (0.7544, 1.09983) (0.756, 1.1424) (0.7576, 1.18268) (0.7592, 1.22092) (0.7608, 1.25732) (0.7624, 1.29204) (0.764, 1.32525) (0.7656, 1.35706) (0.7672, 1.38759) (0.7688, 1.41692) (0.7704, 1.44515) (0.772, 1.47234) (0.7736, 1.49858) (0.7752, 1.52391) (0.7768, 1.54839) (0.7784, 1.57208) (0.78, 1.59501) (0.812, 1.94068) (0.868, 2.27431) (0.924, 2.45743) (0.98, 2.56894) (1.036, 2.64052) (1.092, 2.68754) (1.148, 2.71834) (1.204, 2.7378) (1.26, 2.74895) (1.316, 2.75377) (1.372, 2.75355) (1.428, 2.74922) (1.484, 2.74141) (1.54, 2.73059) (1.596, 2.71708) (1.652, 2.70114) (1.708, 2.68293) (1.764, 2.66258) (1.82, 2.6402) (1.876, 2.61583) (1.932, 2.58951) (1.988, 2.56126) (2.044, 2.53107) (2.1, 2.49894) (2.156, 2.46482) (2.212, 2.42868) (2.268, 2.39044) (2.324, 2.35005) (2.38, 2.30741) (2.436, 2.26241) (2.492, 2.21493) (2.548, 2.16483) (2.604, 2.11193) (2.66, 2.05602) (2.716, 1.99688) (2.772, 1.93421) (2.828, 1.86766) (2.884, 1.79682) (2.94, 1.72117) (2.996, 1.64004) (3.052, 1.55259) (3.108, 1.45768) (3.164, 1.35375) (3.22, 1.23855) (3.276, 1.10858) (3.3, 1.04701) (3.304, 1.03634) (3.308, 1.02554) (3.312, 1.01461) (3.316, 1.00354) (3.32, 0.992334) (3.324, 0.980983) (3.328, 0.969483) (3.332, 0.957828) (3.336, 0.946011) (3.34, 0.934028) (3.344, 0.921871) (3.348, 0.909534) (3.352, 0.897009) (3.356, 0.884288) (3.36, 0.871362) (3.364, 0.858223) (3.368, 0.84486) (3.372, 0.831263) (3.376, 0.817419) (3.38, 0.803317) (3.384, 0.788942) (3.388, 0.774279) (3.392, 0.759311) (3.396, 0.744021) (3.4, 0.728387) (3.404, 0.712387) (3.408, 0.695995) (3.412, 0.679185) (3.416, 0.661923) (3.42, 0.644173) (3.424, 0.625894) (3.428, 0.607038) (3.432, 0.587549) (3.436, 0.567362) (3.44, 0.546401) (3.444, 0.524571) (3.448, 0.50176) (3.452, 0.477827) (3.456, 0.452595) (3.46, 0.425832) (3.464, 0.397229) (3.468, 0.366355) (3.472, 0.332579) (3.476, 0.294905) (3.48, 0.251587) (3.484, 0.198975) (3.488, 0.12591) (3.49, 0)};
\addplot [color=black,line width=0.9pt,smooth] coordinates {(0.7355, 0) (0.7368, -0.16485) (0.7384, -0.381708) (0.74, -0.511707) (0.7416, -0.612917) (0.7432, -0.698019) (0.7448, -0.772395) (0.7464, -0.838936) (0.748, -0.899413) (0.7496, -0.955007) (0.7512, -1.00655) (0.7528, -1.05467) (0.7544, -1.09983) (0.756, -1.1424) (0.7576, -1.18268) (0.7592, -1.22092) (0.7608, -1.25732) (0.7624, -1.29204) (0.764, -1.32525) (0.7656, -1.35706) (0.7672, -1.38759) (0.7688, -1.41692) (0.7704, -1.44515) (0.772, -1.47234) (0.7736, -1.49858) (0.7752, -1.52391) (0.7768, -1.54839) (0.7784, -1.57208) (0.78, -1.59501) (0.812, -1.94068) (0.868, -2.27431) (0.924, -2.45743) (0.98, -2.56894) (1.036, -2.64052) (1.092, -2.68754) (1.148, -2.71834) (1.204, -2.7378) (1.26, -2.74895) (1.316, -2.75377) (1.372, -2.75355) (1.428, -2.74922) (1.484, -2.74141) (1.54, -2.73059) (1.596, -2.71708) (1.652, -2.70114) (1.708, -2.68293) (1.764, -2.66258) (1.82, -2.6402) (1.876, -2.61583) (1.932, -2.58951) (1.988, -2.56126) (2.044, -2.53107) (2.1, -2.49894) (2.156, -2.46482) (2.212, -2.42868) (2.268, -2.39044) (2.324, -2.35005) (2.38, -2.30741) (2.436, -2.26241) (2.492, -2.21493) (2.548, -2.16483) (2.604, -2.11193) (2.66, -2.05602) (2.716, -1.99688) (2.772, -1.93421) (2.828, -1.86766) (2.884, -1.79682) (2.94, -1.72117) (2.996, -1.64004) (3.052, -1.55259) (3.108, -1.45768) (3.164, -1.35375) (3.22, -1.23855) (3.276, -1.10858) (3.3, -1.04701) (3.304, -1.03634) (3.308, -1.02554) (3.312, -1.01461) (3.316, -1.00354) (3.32, -0.992334) (3.324, -0.980983) (3.328, -0.969483) (3.332, -0.957828) (3.336, -0.946011) (3.34, -0.934028) (3.344, -0.921871) (3.348, -0.909534) (3.352, -0.897009) (3.356, -0.884288) (3.36, -0.871362) (3.364, -0.858223) (3.368, -0.84486) (3.372, -0.831263) (3.376, -0.817419) (3.38, -0.803317) (3.384, -0.788942) (3.388, -0.774279) (3.392, -0.759311) (3.396, -0.744021) (3.4, -0.728387) (3.404, -0.712387) (3.408, -0.695995) (3.412, -0.679185) (3.416, -0.661923) (3.42, -0.644173) (3.424, -0.625894) (3.428, -0.607038) (3.432, -0.587549) (3.436, -0.567362) (3.44, -0.546401) (3.444, -0.524571) (3.448, -0.50176) (3.452, -0.477827) (3.456, -0.452595) (3.46, -0.425832) (3.464, -0.397229) (3.468, -0.366355) (3.472, -0.332579) (3.476, -0.294905) (3.48, -0.251587) (3.484, -0.198975) (3.488, -0.12591) (3.49, 0)};
\node at (axis cs: 1.4, -0.35) {\scriptsize{$r_{0}(H,L)$}};
\draw [color=gray, line width=0.3pt] (axis cs: 0.7355, 0)--(axis cs:2.85,3);
\draw [color=gray, line width=0.3pt] (axis cs: 3.49, 0)--(axis cs:3.49,2.2);
\node at (axis cs: 3.5, 3) {\scriptsize{$r_{-}(H,L)$}};
\node at (axis cs: 3.5, 2.5) {\scriptsize{$r_{+}(H,L)$}};
\fill (axis cs: 0.7355, 0) circle (0.8pt);
\fill (axis cs: 3.49, 0) circle (0.8pt);
\end{axis}
\end{tikzpicture}
\captionof{figure}{Qualitative graph of $Z(\cdot;H,L)$ for fixed $(H,L)\in\Lambda$ (on the left) and representation of the corresponding bounded orbit in the $(r,\dot r)$-plane (on the right).}
\label{fig-01}
\end{figure}

From these assumptions it is immediate to see that for every $(H,L)\in \Lambda$ the energy relation \eqref{eq-consenergia3} provides a closed curve in the plane $(r,\dot{r})$ enclosing the point $(r_0(H,L),0)$; moreover, the 	period of the corresponding orbit is given by
\begin{equation} \label{eq-periodoradiale}
T(H,L)=\sqrt{2}\, \int_{r_-(H,L)}^{r_+(H,L)} \dfrac{\mathrm{d}r}{\sqrt{-Z(r; H,L)}}.
\end{equation}
The above quantity can be meant as the \emph{radial period} of a solution of the Hamiltonian system \eqref{eq-sisthamunpert} having energy $H$ and angular momentum $L$, with $(H,L) \in \Lambda$. Taking into account the third equation in \eqref{eq-sisthamunpert}, we deduce that the angular displacement in the radial period (the so-called \emph{apsidal angle}) is
\begin{equation} \label{eq-deltaphi1}
\Theta (H,L) =L\, \int_0^{T(H,L)} \dfrac{\mathrm{d}t}{r(t)^2},
\end{equation}
which can be written, again by means of standard arguments, as
\begin{equation*}
\Theta (H,L) =\sqrt{2}\, L\, \int_{r_-(H,L)}^{r_+(H,L)} \dfrac{\mathrm{d}r}{r^2\sqrt{-Z(1/r;H,L)}},
\end{equation*}
or, equivalently,
\begin{equation} \label{eq-deltaphi2}
\Theta (H,L) =\sqrt{2}\, L\, \int_{1/r_+(H,L)}^{1/r_-(H,L)} \dfrac{\mathrm{d}r}{\sqrt{-Z(1/r;H,L)}}.
\end{equation}
Incidentally, notice that $\Theta(H,L) > 0$ for every $(H,L)\in \Lambda$, since we have assumed $L > 0$.

With this in mind, it is easily seen that a solution of \eqref{eq-sisthamunpert} with energy $H^*$ and angular momentum $L^*$, with $(H^*,L^*) \in \Lambda$, is periodic if and only if there exists a pair of positive coprime integers $(n,k)$ such that
\begin{equation}\label{eq-tori}
\Theta(H^*,L^*) = \frac{2\pi k}{n},
\end{equation}
and, in this case, the minimal period $T$ is given by $T = n T(H^*,L^*)$, cf.~\cite{BDF2}.
In the cartesian coordinate $x = r e^{i\vartheta}$, the integer $k$ is thus the winding number on the time interval $[0,T]$ of the planar curve $x$ around the origin, while the integer $n$ can be meant as the winding number, on $[0,T]$, of the planar curve $(r,\dot r)$ 
(or, equivalently, of the curve $(r,p_r)$) around the point $(r_0(H^*,L^*),0)$.
	
\begin{remark}\label{remark-regolarita}
For further convenience, let us notice that, even if required to be just continuous, the maps $r_0, r_{\pm}$ are actually of class $\mathcal{C}^\infty$, as an easy implicit function argument based on \eqref{eq-ipo1} and \eqref{eq-ipo2} shows. Taking into account this fact, a further implicit function argument together with the differentiability properties of solutions to Cauchy problems (cf.~\cite[Remark~A.1]{BDF2}) yields that the maps $T$ and $\Theta$ are of class $\mathcal{C}^\infty$, as well (for the map $T$, the argument is quite direct; on the other hand, when dealing with $\Theta$ it is convenient to first write it as $\Theta(H,L) = L P(H,L)$, with $P$ a time-map, cf.~\cite[formula~(2.8) and Remark~A.1]{BDF2}).
	
In the case when $\mathcal{H}_0$ is an analytic Hamiltonian, analytic versions of the implicit function theorem \cite[Chapter~6]{KrPa-book} and of the smooth dependence of solutions to Cauchy problems \cite[Chapter~3]{Ha-book} yield the analyticity of the maps $r_0, r_{\pm}$, as well as $T$ and $\Theta$.
\hfill$\lhd$
\end{remark}
	
\begin{remark}\label{remark-subtle}
Taking into account that $\dot r = \alpha(r) p_r$, relation \eqref{eq-consenergia3} can be equivalently written as
\begin{equation}\label{eq-consenergia4}
\dfrac{1}{2}\, p_r^2+\widetilde Z(r; H, L)=0, \qquad \text{where} \quad \widetilde Z(r; H, L) = \frac{Z(r;H,L)}{\alpha(r)^2}.
\end{equation}
Solutions of \eqref{eq-consenergia4}, of course, provide the orbits for the dynamics in the plane $(r,p_r)$, just as \eqref{eq-consenergia3} does for the plane $(r,\dot r)$. In particular, for $(H,L) \in \Lambda$ the period of the orbit passing through $(r_\pm(H, L), 0)$ is given by \eqref{eq-periodoradiale} independently from the plane in which the orbit is considered.

There is however a subtle point which is worth to be emphasized. While the $(r,\dot r)$-plane is more natural form the kinematic point of view, the $(r,p_r)$-plane enjoys the following remarkable property: defining $\mathcal{A}(H,L)$ as the area of the bounded region enclosed by the orbit passing through $(r_\pm(H, L), 0)$ in the $(r,p_r)$-plane, that is, by Gauss--Green formula,
\begin{equation}\label{def-area}
\mathcal{A}(H,L) = 2\sqrt{2}\, \int_{r_-(H,L)}^{r_+(H,L)} \sqrt{-\widetilde Z(r; H,L)}\,\mathrm{d}r,
\end{equation}
the relation
\begin{equation*}
\partial_H \mathcal{A}(H,L) = T(H,L)
\end{equation*}
holds true, as it is easily checked (the analogous property, instead, is not valid in the $(r,\dot r)$-plane). This remark will be crucial in the construction of the action-angle coordinates.
\hfill$\lhd$
\end{remark}

\subsection{Statement and proof}\label{section-3.2}

We are now in a position to state our main result for the Hamiltonian system associated with the Hamiltonian $\mathcal{H}_{\varepsilon}$ defined in \eqref{eq-hamper}, that is
\begin{equation}\label{eq-ham-pert}
\begin{cases}
\, \dot{r} = \alpha(r)\, p_r + \varepsilon \partial_{p_r} \widetilde{\mathcal{H}}(\varepsilon,r,\vartheta,p_r,p_\vartheta),
\vspace{3pt}\\
\, \dot{p_r} = -\dfrac{1}{2}\, \alpha'(r)\, p_r^2 +\dfrac{1}{r^3}\, p_\vartheta^2 + V'(r)
- \varepsilon \partial_{r} \widetilde{\mathcal{H}}(\varepsilon,r,\vartheta,p_r,p_\vartheta), 
\vspace{3pt}\\
\, \dot{\vartheta} = \dfrac{1}{r^2}\, p_{\vartheta}
+ \varepsilon \partial_{p_\vartheta} \widetilde{\mathcal{H}}(\varepsilon,r,\vartheta,p_r,p_\vartheta),
\\
\, \dot{p_\vartheta} = - \varepsilon \partial_{\vartheta} \widetilde{\mathcal{H}}(\varepsilon,r,\vartheta,p_r,p_\vartheta).
\end{cases}
\end{equation}
The statement is as follows.

\begin{theorem}\label{teo-main}
Assume that hypothesis \ref{hp-star} holds true and let $(H^*,L^*) \in \Lambda$ be such that \eqref{eq-tori} is satisfied for a pair of coprime positive integers $(n,k)$. Moreover, suppose that
\begin{equation}\label{eq-nondeg}
\partial_L \Theta(H^*,L^*) \neq 0.
\end{equation}
Then, there exist $\varepsilon^* > 0$ and $\delta^* > 0$ such that,
for every pair of positive coprime integers $(p,q)$ satisfying
\begin{equation}\label{eq-pq}
\left\vert \frac{p}{q} - \frac{k}{n}\right \vert < \delta^*,
\end{equation}
and for every $\varepsilon$ such that $\vert\varepsilon \vert < \varepsilon^*$, there are two distinct periodic solutions 
$(r^j,\vartheta^j,p_r^j,p_\vartheta^j)$ (with $j=1,2$) of the Hamiltonian system \eqref{eq-ham-pert}  with energy $\mathcal{H}_{\varepsilon} = H^*$ and such that, denoting by $T^j$ the minimal period, it holds that:
\begin{itemize}
\item[$(i)$] $\vartheta^j(T^j) -\vartheta^j(0) = 2\pi p$ (that is, the winding number on the time interval $[0,T^j]$ of the planar curve 
$x^j = r^j e^{i \vartheta^j}$ is $p$),
\item[$(ii)$] the winding number on the interval $[0,T^j]$ of the planar curve $(r^j,p_r^j)$
around the point $(r_0(H^*,L^*),0)$ is $q$.
\end{itemize}
\end{theorem}
	
\begin{proof}
As already mentioned, we are going to apply Theorem~\ref{teo-pbapplicato} to the Hamiltonian system obtained from
\eqref{eq-hamper} after passing to action-angle coordinates. 

The transformation to action-angle coordinates 
\begin{equation}\label{def-sigma}
(r,\vartheta,p_r,p_\vartheta) \mapsto \Sigma(r,\vartheta,p_r,p_\vartheta) = (\varphi_1,\varphi_2,I_1,I_2) 
\end{equation}
can be defined in a standard way, analogously to what done in \cite[Section~2]{BDF2} for the case $\alpha \equiv 1$ 
(the only delicate point is that, according to Remark~\ref{remark-subtle}, the area function in the plane $(r,p_r)$, that is \eqref{def-area}, has to be considered, see also \cite{Ar-89, Be-notes}). More precisely, let us consider the set
\begin{equation*}
\left\{ (r,\vartheta,p_r,p_\vartheta) \in \Xi \times \mathbb{R} \times \mathbb{R}^2 \, : \, \dfrac{1}{2}\, p_r^2+\widetilde Z(r; H^*, L^*)=0, \,  p_\vartheta = L^*\right\},
\end{equation*}
where $\widetilde Z$ is defined as in \eqref{eq-consenergia4}. According to the discussion therein, a connected component of the above set is made up by the points $(r,\vartheta,p_r,L^*)$ where $(r,p_r)$ describe the orbit passing through $(r_0(H^*,L^*),0)$: regarding $\vartheta$ as an angular variable, such a connected component is thus a two-dimensional torus, say $\mathcal{T}_0(H^*,L^*)$.
The transformation $\Sigma$ in \eqref{def-sigma} can be defined on an open neighborhood of $\mathcal{T}_0(H^*,L^*)$ and, on this open domain, it transforms system \eqref{eq-hamper} into a system of the type
\begin{equation}\label{eq-sispertazioneangolo2}
\begin{cases}\, \dot{\varphi}=\nabla \mathcal{K}_0 (I)+\varepsilon \, \nabla_I \widetilde{\mathcal{K}}(\varepsilon,\varphi,I),
\vspace{3pt}\\
\, \dot{I}=-\varepsilon \, \nabla_\varphi \widetilde{\mathcal{K}}(\varepsilon,\varphi, I), 
\end{cases}
\end{equation}
where
\begin{equation*}
\mathcal{K}_0(I_1,I_2) = \mathcal{H}_0(\Sigma^{-1}(\varphi_1,\varphi_2,I_1,I_2) )
\end{equation*}
(by the property of the transformation, $\mathcal{K}_0$ does not depend on the angle variables) and
\begin{equation*}
\widetilde{\mathcal{K}}(\varepsilon,\varphi_1,\varphi_2,I_1,I_2) = \widetilde{\mathcal{H}}(\varepsilon,\Sigma^{-1}(\varphi_1,\varphi_2,I_1,I_2) ).
\end{equation*}
Since $\Sigma$ is of class $\mathcal{C}^\infty$, the transformed Hamiltonians $\mathcal{K}_0$ and $\widetilde{\mathcal{K}}$ are $\mathcal{C}^\infty$, as well.

With this transformation, the torus $\mathcal{T}_0(H^*,L^*)$ is transformed into the set
\begin{equation*}
\{ I_1 = I_1^*, \, I_2 = I_2^*\},
\end{equation*}
where $(I_1^*,I_2^*)$ are the actions corresponding to $(H^*,L^*)$ via the one-to-one map
\begin{equation*}
\Psi \colon (H,L) \mapsto (I_1,I_2) = \left(\frac{1}{2\pi} \mathcal{A}(H,L) + L, L \right).
\end{equation*}
Moreover, as proved in \cite[Proof of Theorem~3.1]{BDF2},
\begin{equation*}
\partial_{I_1}\mathcal{K}_0(I_1,I_2) = \frac{2\pi}{T(\Psi^{-1}(I_1,I_2))}, \qquad \partial_{I_2}\mathcal{K}_0(I_1,I_2) = \frac{\Theta(\Psi^{-1}(I_1,I_2))-2\pi}{T(\Psi^{-1}(I_1,I_2))},
\end{equation*}
where $T$ and $\Theta$ are as in \eqref{eq-periodoradiale} and \eqref{eq-deltaphi1}. Thus
\begin{equation*}
\frac{\partial_{I_2}\mathcal{K}_0(I_1,I_2)}{\partial_{I_1}\mathcal{K}_0(I_1,I_2)} = \frac{1}{2\pi} \Theta((\Psi^{-1}(I_1,I_2))) - 1
\end{equation*}
and so, by \eqref{eq-tori},
\begin{equation*}
\frac{\partial_{I_2}\mathcal{K}_0(I_1^*,I_2^*)}{\partial_{I_1}\mathcal{K}_0(I_1^*,I_2^*)} = \frac{k-n}{n}.
\end{equation*}
Hence, \eqref{eq-ipoazioneangolo} is satisfied for $n_2 = k-n \in \mathbb{Z}$ and $n_1 = n \in \mathbb{N} \setminus \{0\}$ (notice that $n_1$ and $n_2$ are coprime, since $k$ and $n$ are so).

It thus remains to check that the non-degeneracy condition \eqref{eq-iponondegenere} is satisfied, as well. To this end, we take advantage of the discussion in Remark~\ref{rem-alternativa}: accordingly, we are going to prove that the function
\begin{equation*}
\eta(I_2) := \frac{\partial_{I_2}\mathcal{K}_0(\mathcal{I}_0(I_2),I_2)}{\partial_{I_1}\mathcal{K}_0(\mathcal{I}_0(I_2),I_2)} = 
\frac{1}{2\pi}\Theta(\Psi^{-1}(\mathcal{I}_0(I_2),I_2))-1
\end{equation*}
has non-zero derivative at $I_2^* = L^*$, recalling that $\mathcal{I}_0(I_2)$ is defined by
$\mathcal{K}_0(\mathcal{I}_0(I_2),I_2) = H^*$. To this end, we notice that
\begin{equation*}
\Psi^{-1}(\mathcal{I}_0(I_2),I_2) = (H^*,I_2)
\end{equation*}
and so 
\begin{equation*}
\eta(I_2) = \frac{1}{2\pi}\Theta(H^*,I_2)-1.
\end{equation*}
Thus, $\eta'(I_2^*) \neq 0$ as a direct consequence of \eqref{eq-nondeg} and of the fact that $I_2^* = L^*$.

Summing up, all the assumptions of Theorem~\ref{teo-pbapplicato} are satisfied. Accordingly, there exist $\varepsilon^* >0$ and $\delta^* > 0$ such that, for every pair of coprime integers $(m_1,m_2) \in (\mathbb{N}\setminus\{0\}) \times \mathbb{Z}$ satisfying
\begin{equation}\label{hpm1}
\left\vert \frac{m_2}{m_1} - \frac{k-n}{n}\right \vert < \delta^*,
\end{equation}
and for every $\varepsilon$ satisfying $\vert\varepsilon \vert < \varepsilon^*$, there are two distinct periodic solutions of the Hamiltonian system \eqref{eq-sispertazioneangolo2} with energy $\mathcal{K}_{\varepsilon} = H^*$ and winding 
vector on their minimal period equal to  $(l_1,l_2) =(m_1,m_2)$. Thus, if $(p,q)$ is a pair of positive coprime integers satisfying \eqref{eq-pq}
one can choose $m_2 = p-q$ and $m_1 = q$; by undoing the change of variables, two distinct periodic solutions of the Hamiltonian systems \eqref{eq-hamper} of energy $\mathcal{H}_{\varepsilon} = H^*$ are obtained. Information $(i)$ and $(ii)$ about the winding numbers of these solutions are a consequence of the way in which the angle variables $\varphi_1, \varphi_2$ are defined in the transformation $\Sigma$,
see \cite[Section~2]{BDF2} for more details. 
\end{proof}

\begin{remark}
By \eqref{eq-tori} and \eqref{eq-nondeg}, for any pair of positive coprime integers $(p,q)$ satisfying \eqref{eq-pq} there exists a unique value $L$ near $L^*$ such that
\begin{equation*}
\Theta(H^*,L) = \frac{2\pi p}{q}.
\end{equation*}
By a careful inspection of the proof of Theorem~\ref{teo-main}, and taking into account Remark~\ref{rem2},
we can deduce that, for the solutions corresponding to $(p,q)$, the coordinate $p_\vartheta^j$ is arbitrarily near, as $\varepsilon \to 0$, to such $L$.
\hfill$\lhd$
\end{remark}

We conclude this section by presenting a couple of applications of Theorem~\ref{teo-main} to perturbations of central force problems in the plane, namely
\begin{equation}\label{eq-cenfor}
\ddot x = V'(\vert x \vert) \frac{x}{\vert x \vert} + \varepsilon \,\nabla U(x), \qquad x \in \mathbb{R}^2 \setminus \{0\},
\end{equation}
where $V \colon (0,+\infty) \to \mathbb{R}$ and $U \colon \mathbb{R}^2 \setminus \{0\} \to \mathbb{R}$ are functions of class $\mathcal{C}^\infty$. 
As well known, by passing to polar coordinates $x = r e^{i\vartheta}$, the unperturbed Hamiltonian $\mathcal{H}_0$ writes as
\begin{equation}\label{cenfor}
\mathcal{H}_0(r,\vartheta,p_r,p_\vartheta)=\displaystyle \dfrac{1}{2}\, \left(p_{r}^2+\dfrac{1}{r^2} \, p_{\vartheta}^2\right) - V(r),
\quad (r,\vartheta,p_r,p_\vartheta)\in (0,+\infty) \times \mathbb{R} \times \mathbb{R}^2,
\end{equation}
which is of the form \eqref{eq-hamunpertastratta} with $\Xi = (0,+\infty)$ and $\alpha \equiv 1$.

In what follows, we exhibit two specific choices of $V$ for which non-circular periodic solutions of the unperturbed problem exist and the non-degeneracy condition \eqref{eq-nondeg} is satisfied: thus, in both the cases Theorem~\ref{teo-main} can be applied to ensure bifurcation of closed orbits for \eqref{eq-cenfor}, when $\varepsilon$ is small enough. Our discussion below will take advantage of considerations and computations already presented in \cite{BDF2}, where the fixed period problem was investigated.

\begin{example}[Levi-Civita potential]\label{example-relativistic}
Let us consider the potential $V(r) = \kappa/r + \lambda/r^2$ for $\kappa,\lambda > 0$. The associated central force problem is
\begin{equation*}
\ddot x = -\kappa \dfrac{x}{|x|^3} - 2\lambda \dfrac{x}{|x|^4},
\end{equation*}
which was proposed by Levi-Civita \cite{LeCi-28} as a relativistic correction for the Kepler problem, see~\cite{BDF4} and the references therein.
As shown in \cite[Section~4.1]{BDF2}, the natural domain for the existence of non-circular closed orbits is
\begin{equation*}
\Lambda = \biggl{\{}(H,L)\in\mathbb{R}^{2} \colon L\in (\sqrt{2\lambda},+\infty), \; H\in \left(-\dfrac{\kappa^{2}}{2(L^{2}-2\lambda)}, 0\right) \biggr{\}}.
\end{equation*}
and, on this set,
\begin{equation*}
\Theta(H,L) = \dfrac{2\pi L}{\sqrt{L^{2}-2 \lambda}}.
\end{equation*}
Hence,
\begin{equation*}
\partial_{L} \Theta(H,L) = -\dfrac{4 \pi \lambda}{(L^{2}-2 \lambda)^{\frac{3}{2}}} <0,
\end{equation*}
so that Theorem~\ref{teo-main} can be applied.
We notice that in this way we obtain a generalization of \cite[Theorem~3.1]{BDF4}, where bifurcation from a fixed invariant torus was consider (cf.~the discussion at the end of \cite[Remark~1]{BDF4}).

Incidentally, we notice that the above formulas are valid also in the case $\lambda=0$ (i.e.~the Kepler problem) leading to the well-known fact that $\Theta(H,L)\equiv 2\pi$. Hence, in this case Theorem~\ref{teo-main} does not apply. Actually for the Kepler problem a result of this type is expected to be false (see the discussion in \cite{BoDaFe-21,BoOrZh-19} for more details).
\end{example}

\begin{example}[Homogeneous potentials]\label{example-hom-potentials}
Let us consider the homogeneous potential $V(r) = \kappa/(\alpha r^\alpha)$, where $\kappa>0$ and $\alpha \neq 0$.
The associated central force problem is
\begin{equation}\label{eq-hom}
\ddot x = -\kappa \dfrac{x}{|x|^{\alpha + 2}},
\end{equation}
which, as well known, admits non-circular periodic solutions if and only if $\alpha<2$, cf.~\cite[p.~7]{AmCo-93}. 
As shown in \cite[Section~4.2]{BDF2}, the natural domain for the existence of non-circular closed orbits is
\begin{equation*}
\Lambda = \biggl{\{}(H,L)\in\mathbb{R}^{2} \colon L\in (0,+\infty), \; H\in \left(-\dfrac{2-\alpha}{2\alpha}\kappa^{\frac{2}{2-\alpha}} L^{-\frac{2\alpha}{2-\alpha}}, 0\right) \biggr{\}}.
\end{equation*}
The map $\Theta(H,L)$ does not admit an explicit expression for a general value of $\alpha$, however, as discussed in \cite[Remark~4.2]{BDF2}, it is possible to prove that
\begin{equation*}
\partial_{L} \Theta(H,L)
\begin{cases}
>0, & \text{if $\alpha\in(-\infty,-2)\cup(0,1)$,}
\\
<0, & \text{if $\alpha\in(-2,0)\cup(1,2)$,}
\end{cases}
\end{equation*}
for all $(H,L)\in\Lambda$. Hence, Theorem~\ref{teo-main} applies for every $\alpha\in(-\infty,2)\setminus\{-2,0,1\}$. 

Let us notice that in the case $\alpha=0$ equation \eqref{eq-hom} still makes sense, coming from the logarithmic potential $V(r)=-\kappa \log(r)$. According to \cite{Ca-14} (see also \cite[Remark~4.3]{BDF2}), $\partial_{L} \Theta(H,L)>0$ and so again Theorem~\ref{teo-main} applies. The only excluded cases are thus $\alpha=-2$ and $\alpha=1$, corresponding respectively to the case of the Kepler problem and to the harmonic oscillator.
\end{example}

\section{Closed geodesics in a perturbed Schwarzschild metric}\label{section-sch}

In this section, we apply the general results of Section~\ref{section-main} to the study of periodic geodesic motions for perturbations of the Schwarzschild metric in general relativity. More precisely, we are going to investigate bifurcation of equatorial timelike geodesics for a stationary and equatorial perturbation of the Schwarzschild metric.

This section is organized as follows. In Section~\ref{sub-hamschw} we introduce the Hamiltonian formulation for the unperturbed model and we prove that, for suitable choices of the values of its first integrals, radially closed orbits exist.
In Section~\ref{section-4.2} we investigate the non-degeneracy of such orbits. Finally, in Section~\ref{section-4.3} we prove bifurcation of closed orbits for the perturbed problem, via an application of Theorem~\ref{teo-main}.

\subsection{The unperturbed problem: Hamiltonian formulation and equatorial timelike closed geodesics}\label{sub-hamschw}
	
We follow the notation and the presentation in \cite{Xue-21}.
Let us consider the $4$-dimensional space-time endowed with spherical coordinates 
\begin{equation}\label{sph-coord}
q = (q^0,q^1,q^2,q^3)=(\tau,r,\theta,\phi)\in \mathbb{R} \times (0,+\infty)\times [0,\pi]\times [0,2\pi),
\end{equation}
where $\tau$ is the coordinate time, $r$ is the polar radius, $\theta$ is the inclination angle (colatitude) and $\phi$ is the azimuthal angle.
In geometrized units ($G=c=1$) the Schwarzschild metric $g$ is given by
\begin{equation*}
\mathrm{d}s^{2} = g_{\mu \nu} \, \mathrm{d}q^\mu \mathrm{d}q^\nu = -\alpha(r)\, \mathrm{d}\tau^2+\alpha(r)^{-1}\, \mathrm{d}r^2+r^2\, \mathrm{d}\theta^2+r^2\sin^2\theta \, \mathrm{d}\phi^2,
\end{equation*}
where 
\begin{equation}\label{eq-defalpha}
\alpha(r)=1-\dfrac{2M}{r},
\end{equation}
with $M>0$ the mass of the blackhole.
The surface $\{r=2M\}$ is known as event horizon: in the region inside this surface, both light and particles are forced to the singularity $ r = 0$. Hence, in what follows we assume $r>2M$, and so $\alpha(r)>0$.

The Lagrangian $L_0$ of the free particle is 
\begin{equation}\label{eq-lagrunpert}
L_0 (q,\dot{q})= \dfrac{1}{2} g_{\mu \nu} \dot{q}^\mu \dot{q}^\nu = \dfrac{1}{2} \left(-\alpha(r)\, \dot{\tau}^2+\alpha(r)^{-1}\, \dot{r}^2+r^2 \, \dot{\theta}^2+r^2\sin^2 \theta \, \dot{\phi}^2\right),
\end{equation}
where the derivative is taken with respect to the proper time $t$ of the particle. 
By the classical Legendre transformation, the Hamiltonian $H_0$ conjugate to $L_0$ is given by
\begin{equation} \label{eq-hamrunpert}
H_0 (q,p)= \dfrac{1}{2} \left(-\alpha(r)^{-1}\, p_{\tau}^2+\alpha(r)\, p_{r}^2+\dfrac{1}{r^2} \, p_{\theta}^2+\dfrac{1}{r^2\sin^2 \theta} \, p_{\phi}^2\right),
\end{equation}
where $p=(p^{0}, p^1,p^2,p^3)=(p_{\tau},p_r,p_\theta,p_\phi)$ are the canonical momenta, i.e.~$p^{\mu} = \partial_{q^{\mu}} L_{0}$, for $\mu=0,1,2,3$.

The associated Hamiltonian system is then given by
\begin{equation} \label{eq-sisthamschwunpert}
\begin{cases}
\, \dot{\tau} = \displaystyle -\alpha(r)^{-1}\, p_\tau, \quad 
&\dot{p_\tau} = 0,
\vspace{3pt}\\
\, \dot{r} = \displaystyle\alpha(r)\, p_r,
&\dot{p_r} =  \displaystyle -\dfrac{1}{2}\, \alpha'(r)\, \alpha(r)^{-2} \, p_\tau^2-\dfrac{1}{2}\, \alpha'(r)\, p_r^2 +\dfrac{1}{r^3}\, p_\theta^2+\dfrac{1}{r^3 \sin^2 \theta}\, p_\phi^2,
\vspace{3pt}\\
\, \dot{\theta} =  \displaystyle \dfrac{1}{r^2}\, p_\theta, 
& \dot{p_\theta} =  \displaystyle \dfrac{1}{r^2 \sin^3 \theta}\, \cos \theta \, p_\phi^2,
\vspace{3pt}\\
\, \dot{\phi} = \displaystyle \dfrac{1}{r^2 \sin^2 \theta}\, p_{\phi}, 
&\dot{p_\phi} = 0.
\end{cases}
\end{equation}
Since $\dot{p_\tau} = 0$, we immediately deduce that $p_\tau$ is a first integral, the so-called particle energy; moreover, it is easy to deduce that $(\theta,p_\theta) \equiv (\pi/2,0)$ satisfies the pair of equations for $(\dot{\theta},\dot{p_\theta})$. As a consequence, denoting by $E\in \mathbb{R}$ the value of the first integral $p_\tau$, we can consider the $2$-degree of freedom system associated with the reduced Hamiltonian
\begin{equation} \label{eq-hamrunpertschwridotta0}
{\mathcal{H}_0} (r,\phi,p_r,p_\phi)= \dfrac{1}{2} \left(\alpha(r)\, p_{r}^2+\dfrac{1}{r^2} \, p_{\phi}^2\right)-\dfrac{1}{2}\, \dfrac{E^2}{\alpha(r)}.
\end{equation}
which describes the \textit{equatorial} solutions ($\theta \equiv \pi/2$) of \eqref{eq-sisthamschwunpert}.

From now on, we will focus on solutions satisfying the isoenergetic condition
\begin{equation} \label{eq-schwisoen1}
\mathcal{H}_0 (r,\phi,p_r,p_\phi)=-\dfrac{1}{2}.
\end{equation}
This means that the corresponding Lagrangian $L_{0}$ is fixed to be $-1/2$, namely
\begin{equation*}
g_{\mu \nu} \dot{q}^\mu \dot{q}^\nu=-1.
\end{equation*}
From a physical point of view this corresponds to the case of massive particles, that is \textit{timelike geodesics}.
We refer to Remark~\ref{rem-massless} for some consideration on the massless case (i.e.~lightlike geodesics).

Let us notice now that the hamiltonian \eqref{eq-hamrunpertschwridotta0} is of the form \eqref{eq-hamunpertastratta}, with $\alpha$ as in \eqref{eq-defalpha}, 
\begin{equation*}
V(r) = \dfrac{1}{2}\, \dfrac{E^2}{\alpha(r)},
\end{equation*}
and the role of the angle variable $\vartheta$ being played by the azimuthal angle $\phi$. Moreover, the effective potential defined in \eqref{eq-defW} is here given by
\begin{equation}\label{eq-potschw}
Z(r;H,L)=\left(\dfrac{1}{2}\, \dfrac{L^2}{r^2}-H\right)\alpha(r)-\dfrac{1}{2}\, E^2=W(r; H,L)-\dfrac{1}{2}\, E^2,
\end{equation}
where $H$ and $L$ denotes as usual the energy and the angular momentum and
\begin{equation*}
W(r;H,L)=\dfrac{1}{2}\, \left(\dfrac{L^2}{r^2}-2H\right)\alpha(r).
\end{equation*}

The rest of the section is devoted to show that the global picture of the radial dynamics under this effective potential is, for suitable values of $E$, the one discussed in Section~\ref{sec-prelastrattononpert}.
According to the isoenergetic condition \eqref{eq-schwisoen1}, we are interested in the case $H=-1/2$; hence, we can limit our analysis to the case $H<0$. On the other hand, we recall that without loss of generality $L>0$ (cf.~Section~\ref{sec-prelastrattononpert}).

Let us define
\begin{equation} \label{eq-deflambdatilda}
\widetilde{\Lambda} = \bigl{\{}(H,L)\in \mathbb{R}^2 \colon H<0,\, L>4\sqrt{2}\, \sqrt{-H}\, M \bigr{\}}.
\end{equation}
Some properties of the potential $W$ are collected in the next Lemma, whose proof is straightforward. See also Figure~\ref{fig-02} for a graphical representation of $W(\cdot; H,L)$.

\begin{lemma} \label{lem-potenzialeW1}
For every $(H,L)\in \widetilde{\Lambda}$, the map $W(\cdot; H,L)\colon (2M,+\infty)\to \mathbb{R}$ satisfies the following properties:
\begin{itemize}
\item[$(i)$] the limit relation
\begin{equation} \label{eq-limvpotschw}
\lim_{r\to +\infty} W(r;H,L)=-H
\end{equation}
holds true;
\item[$(ii)$] $W(\cdot; H,L)$ has exactly two critical points: precisely, a maximum point 
\begin{equation*}
r^* (H,L)=\dfrac{L^2- L\sqrt{L^2+24HM^2}}{-4HM}
\end{equation*}
and a minimum point
\begin{equation} \label{eq-defrpiu}
r_0 (H,L)=\dfrac{L^2+ L\sqrt{L^2+24HM^2}}{-4HM};
\end{equation}
\item[$(iii)$] the minimum and the maximum points of $W(\cdot ; H,L)$ satisfy 
\begin{equation*}
6M < r^*(H,L) < r_0(H,L)
\end{equation*}
and
\begin{equation} \label{eq-vpotschwmax}
0<W(r_0(H,L);H,L)<-H<W(r^*(H,L);H,L);
\end{equation}
\item[$(iv)$] the minimum point $r_0(H,L)$ is non-degenerate, i.e. 
\begin{equation} \label{eq-r0nondeg}
W''(r_0(H,L); H,L)>0.
\end{equation}
\end{itemize}
\end{lemma}	

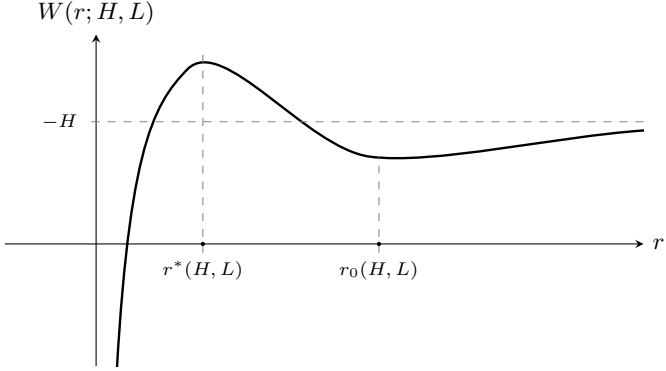
\begin{figure}[htb]
\definecolor{bdf-purple}{RGB}{85,0,130}
\definecolor{bdf-yellow}{RGB}{255,183,0}
\begin{tikzpicture}
\begin{axis}[
  tick label style={font=\scriptsize},
  axis y line=middle, 
  axis x line=middle,
  xtick={0},
  ytick={0},
  xticklabels={$0$},
  yticklabels={$0$},
  xlabel={\small $r$},
  ylabel={\small $W(r;H,L)$},
every axis x label/.style={
    at={(ticklabel* cs:1.0)},
    anchor=west,
},
every axis y label/.style={
    at={(ticklabel* cs:1.0)},
    anchor=south,
},
  width=10cm,
  height=6cm,
  xmin=-1,
  xmax=6,
  ymin=-0.7,
  ymax=1.2]
  \addplot [color=black,line width=0.9pt,smooth] coordinates {(0.05, -10) (0.2, -1) (1,1) (3,0.5) (6,0.65) (10,0.75)};
\draw [color=gray, dashed, line width=0.3pt] (axis cs:1.17,-0.05)--(axis cs:1.17,1.1);
\draw [color=gray, dashed, line width=0.3pt] (axis cs:3.1,-0.05)--(axis cs:3.1,0.5);
\draw [color=gray, dashed, line width=0.3pt] (axis cs:-0.05,0.7)--(axis cs:7,0.7);
\node at (axis cs: 1.17,-0.15) {\scriptsize{$r^*(H,L)$}};
\node at (axis cs: 3.1,-0.15) {\scriptsize{$r_0(H,L)$}};
\node at (axis cs: -0.4,0.7) {\scriptsize{$-H$}};
\fill (axis cs: 1.17,0) circle (0.8pt);
\fill (axis cs: 3.1,0) circle (0.8pt);
\end{axis}
\end{tikzpicture}
\captionof{figure}{Qualitative graph of $W(\cdot;H,L)$ for fixed $(H,L)\in \widetilde{\Lambda}$.}
\label{fig-02}
\end{figure}

Let us now consider, for $(H,L)\in \widetilde{\Lambda}$, the quantity
\begin{equation} \label{eq-defomegazero}
\omega_0(H,L)=-W(r_0(H,L);H,L).
\end{equation}
namely the opposite of the minimum value of $W(\cdot;H,L)$, and accordingly let us set
\begin{equation*}
\Lambda_E =\left\{(H,L)\in \widetilde{\Lambda}\colon -\omega_0(H,L)<\dfrac{1}{2}\, E^2 < -H\right\},
\end{equation*}
where we recall that by now $E$ is fixed in an arbitrarily way.
In Proposition~\ref{lem-schworbitechiuse} we are going to prove that when $E\in(\sqrt{{25}/{27}},1)$ the set $\Lambda_E$ is not empty and, for $(H,L)\in\Lambda_E$ the dynamics in the $(r,\dot{r})$-plane fits into the abstract framework of Section~\ref{sec-prelastrattononpert}.

To simplify the coming computation, let us introduce the auxiliary function $\zeta\colon (4\sqrt{2}, +\infty)\to \mathbb{R}$ defined by
\begin{equation} \label{eq-defzetaausiliaria}
\zeta (u)= \dfrac{u^4-20u^2+32+u(u^2-8)\sqrt{u^2-24}}{u^4-18u^2+u(u^2-6)\sqrt{u^2-24}}, \quad  u\in(4\sqrt{2},+\infty).
\end{equation}
It is straightforward to check that $\zeta$ satisfies 
\begin{equation} \label{eq-propzetaausiliaria}
\zeta' (u)>0, \; \text{for every $u\in(4\sqrt{2},+\infty)$,} \qquad \zeta \bigl( (4\sqrt{2}, +\infty) \bigr)=\left(\dfrac{25}{27},1\right).
\end{equation}
In particular, this proves that $\zeta$ is invertible on its domain.
The next technical lemma highlights a hidden homogeneity property.

\begin{lemma}
For every $(H,L)\in \widetilde{\Lambda}$, we have
\begin{equation} \label{eq-v0maxrif}
\omega_0(H,L)=H\, \zeta \left(\dfrac{L}{\sqrt{-H}\, M}\right).
\end{equation}
\end{lemma}

\begin{proof}
Let us fix $(H,L)\in \widetilde{\Lambda}$. From the fact that $r_0 (H,L)$ is a critical point of $W$ we deduce that 
\begin{equation*}
-2HM\, (r_0(H,L))^2-L^2 r_0(H,L)+3ML^2=0,
\end{equation*}
i.e. 
\begin{equation*}*\label{eq-quadratomax}
(r_0(H,L))^2=\dfrac{L^2 r_0 (H,L)-3ML^2}{-2HM}.
\end{equation*}
This implies that
\begin{equation*}
\dfrac{L^2}{(r_0(H,L))^2}=-2H\, \dfrac{r_0 (H,L)- 2M}{r_0(H,L)-3M}
\end{equation*}
and then, recalling \eqref{eq-defalpha} and \eqref{eq-defomegazero},
\begin{equation*}
\omega_0(H,L)=H \, \dfrac{r_0(H,L) - 2M}{r_0(H,L)-3M}\, \dfrac{r_0(H,L) - 2M}{r_0(H,L)}.
\end{equation*}
Taking into account \eqref{eq-defrpiu} and \eqref{eq-defzetaausiliaria}, a straighforward computation proves that \eqref{eq-v0maxrif} is satisfied.
\end{proof}

We are now in a position to state and prove the final result of the section, giving a more explicit description of the set $\Lambda_E$ and proving that the choice $\Lambda=\Lambda_E$ is admissible for hypothesis \ref{hp-star} (in particular, pairs $(H,L)\in \Lambda_E$ give rise to solutions of \eqref{eq-consenergia3} which are periodic in the radial component).

\begin{proposition} \label{lem-schworbitechiuse}
Let $E\in (\sqrt{25/27},1)$. Then, the set $\Lambda_E$ can be written as
\begin{equation*}
\Lambda_E=\left\{(H,L)\in \mathbb{R}^2 \colon \begin{array}{l}-\dfrac{27}{50}\, E^2<H<-\dfrac{1}{2}\, E^2,
\\ 4\sqrt{2}\, \sqrt{-H}\, M < L < \sqrt{-H}\, M\, \zeta^{-1} \left(\dfrac{E^2}{-2H}\right)
\end{array}
\right\};
\end{equation*}
Moreover, 
\begin{itemize}
\item[$(i)$] the function $r_{0}$ defined in \eqref{eq-defrpiu} satisfies \eqref{eq-ipo1};
\item[$(ii)$] there exists two continuous functions $r_-,r_+\colon \Lambda_{E}\to(0,+\infty)$ with $r_-<r_0<r_+$ and satisfying \eqref{eq-ipo2};
\end{itemize}
where $Z$ is as in \eqref{eq-potschw}.
\end{proposition}

\begin{proof}
Recalling \eqref{eq-deflambdatilda} and \eqref{eq-v0maxrif}, we have 
\begin{equation} \label{eq-lambdaErif}
\Lambda_E=\left\{(H,L)\in \mathbb{R}^2\colon L>4\sqrt{2}\, \sqrt{-H}\, M, \, H<-\dfrac{1}{2}\, E^2, \, \zeta \left(\dfrac{L}{\sqrt{-H}\, M}\right)<\dfrac{E^2}{-2H}
\right\}.
\end{equation}
From the fact that $\zeta ((4\sqrt{2}, +\infty))=\left(25/27,1\right)$ (cf.~\eqref{eq-propzetaausiliaria}), we infer that
\begin{equation*}
\dfrac{E^2}{-2H} > \dfrac{25}{27}, \quad \text{for every $(H,L)\in \Lambda_E$,}
\end{equation*}
and so
\begin{equation*}
H>-\dfrac{27}{50}\, E^2, \quad \text{for every $(H,L)\in \Lambda_E$.}
\end{equation*}
Notice now that, when the above condition is fulfilled, from \eqref{eq-lambdaErif} we obtain
\begin{equation*}
L < \sqrt{-H}\, M\, \zeta^{-1} \left(\dfrac{E^2}{-2H}\right)
\end{equation*}
and this conclude the proof of the first part of the statement.

As for the second part, point $(ii)$ directly follows from \eqref{eq-r0nondeg}, recalling the definitions of $r_0$ and $\Lambda_E$. On the other hand, let us observe that the definition of $\Lambda_E$ together with \eqref{eq-limvpotschw} and \eqref{eq-vpotschwmax} imply that for every $(H,L)\in \Lambda_E$ the equation 
\begin{equation*}
W(r;H,L)=\dfrac{1}{2}\, E^2
\end{equation*}
has exactly two solutions $r_\pm (H,L)$ with $r^{*}(H,L) < r_- (H,L) < r_0 (H,L) < r_+ (H,L)$. Moreover, $W'(r;H,L)(r-r_0(H,L))>0$, for every $r\in [r_-(H,L),r_+(H,L)]\setminus\{r_{0}(H,L)\}$, since the only critical point of $W(\cdot,H,L)$ in $[r_-(H,L),r_+(H,L)]$ is $r_{0}(H,L)$, as observed in Lemma~\ref{lem-potenzialeW1}.
This also ensures that $r_\pm$ are continuous functions. Recalling \eqref{eq-potschw}, we conclude that \eqref{eq-ipo2} is fulfilled.
\end{proof}

\begin{remark} \label{rem-altrocaso}
Let us observe that the effective potential $W(\cdot; H,L)$ has a minimum point if and only if $L>2\sqrt{3}\, \sqrt{-H}\, M$. As a consequence, bounded orbits exist just for these values of the angular momemtum. In the previous discussion, for simplicity, we have considered the case $L>4\sqrt{2}\, \sqrt{-H}\, M$; when $2\sqrt{3}\, \sqrt{-H}\, M<L \leq 4\sqrt{2}\, \sqrt{-H}\, M$ the only difference is that the maximum of the potential $W(\cdot; H,L)$ is less or equal than the asymptotic value $-H$ (contrary to \eqref{eq-vpotschwmax}). This implies that a more careful choice of the values $E$ giving rise to bounded orbits has to be made; however, for these values it is possible to prove an analogous of Lemma~\ref{lem-schworbitechiuse}. We leave the details to the reader.
\hfill$\lhd$
\end{remark}

Recalling now that we are interested in periodic solutions satisfying the isoenergetic condition \eqref{eq-schwisoen1}, we observe that
\begin{equation}\label{eq-defLE0}
\Lambda_E\cap (\{-1/2\}\times \mathbb{R}) =(4M,L_E),
\end{equation}
where $L_E$ is given by
\begin{equation*}
L_E=\dfrac{M}{\sqrt{2}}\, \zeta^{-1} (E^2).
\end{equation*}

\subsection{The unperturbed problem: the non-degeneracy condition}\label{section-4.2}

According to Proposition~\ref{lem-schworbitechiuse}, we are allowed to consider the map $\Theta\colon \Lambda_E\to \mathbb{R}$ defined in \eqref{eq-deltaphi2}. Recalling \eqref{eq-potschw}, we thus have
\begin{equation} \label{eq-Theta2}
\Theta (H,L) =\sqrt{2}\, L\, \int_{1/r_+(H,L)}^{1/r_-(H,L)} \dfrac{\mathrm{d}r}{\sqrt{E^2/2-W(1/r;H,L)}},
\end{equation}
for every $(H,L)\in \Lambda_E$.
Notice that, according to Remark~\ref{remark-regolarita}, the map $\Theta$ is analytic on the domain $\Lambda_E$; in particular, recalling \eqref{eq-defLE0}, the one-variable function $\Theta(-1/2,\cdot)$ is analytic on the interval $(4M,L_E)$.

The aim of this section is to prove that the non-degeneracy condition
\begin{equation}\label{eq-nondeg-Schw}
\partial_L \Theta\biggl{(}-\dfrac{1}{2},L^*\biggr{)} \neq 0
\end{equation}
required for the application of Theorem~\ref{teo-main}, 
is satisfied for a dense set of values $L^{*}$ in the interval $(4M,L_E)$, cf.~Proposition~\ref{teo-schwnonperturbato}.

The strategy will be the following: at first, we prove, via some time-map arguments along the lines of \cite[Appendix]{BDF2}, that 
\begin{equation*}
\lim_{L\to (L_E)^-} \partial_L \Theta \biggl{(}-\dfrac{1}{2},L\biggr{)} \neq 0.
\end{equation*}
Secondly, we exploit the analyticity of $\Theta$ to reach the conclusion.

We start by introducing some preliminary notation.
First of all, let us notice that from \eqref{eq-Theta2} we recognize that $\Theta$ can be written as
\begin{equation} \label{eq-thetaP}
\Theta (H,L)=L\, P(H,L),\quad \text{for all $(H,L)\in \Lambda_E$,}
\end{equation}
where $P$ is the period map of the nonlinear oscillator $\tfrac{1}{2} \dot{r}^{2} + \widetilde{W}(r; H,L) = 0$, where
\begin{equation*}
\widetilde{W}(r; H,L)=W(1/r;H,L).
\end{equation*} 
From Lemma~\ref{lem-potenzialeW1} it is immediate to see that $\widetilde{W}(\cdot;H,L)$ has a local non-degenerate minimum point at $1/r_0(H,L)$, for every $(H,L)\in \Lambda_E$. The corresponding minimum is $-\omega_0(H,L)$, where $\omega_0(H,L)$ is given in \eqref{eq-defomegazero}. 

Now, for $E\in (\sqrt{25/27},1)$, let us define the set
\begin{equation*}
\Gamma_E=\bigl{\{}(r;H,L)\in\mathbb{R}^{3} \colon (H,L)\in \Lambda_E, \; r\in (1/r_+(H,L),1/r_-(H,L))\bigr{\}}
\end{equation*}
and the function $\Omega\colon \Gamma_E\to \mathbb{R}$ by 
\begin{equation} \label{eq-defomegatrasl}
\Omega (r; H,L)=\widetilde{W}(r; H,L)+\omega_0(H,L).
\end{equation}
The map $P$ can then be written as
\begin{equation*}
P(H,L) = \sqrt{2} \int_{1/r_+(H,L)}^{1/r_-(H,L)} \dfrac{\mathrm{d}r}{\sqrt{E^2/2+\omega_0 (H,L)-\Omega(r;H,L)}},\quad (H,L)\in \Lambda_E.
\end{equation*}
We now define the map
\begin{equation*}
h(r;H,L) = \mathrm{sgn} (r-1/r_{0}(H,L)) \sqrt{\Omega(r;H,L)}, \quad (r,H,L)\in \Gamma_E.
\end{equation*}
By adapting the arguments in \cite[Appendix]{BDF2}, it is possible to show that $h(\cdot;H,L)$ is invertible in $(1/r_+(H,L),1/r_-(H,L))$, for every $(H,L)\in \Lambda_E$. Introducing the change of variable
\begin{equation} \label{eq-cambiovar}
r=r(\theta;H,L):=h^{-1}(\sqrt{E^2/2+\omega_{0}(H,L)}\sin (\theta; H,L),
\quad \theta \in \biggl{[}-\frac{\pi}{2},\frac{\pi}{2}\biggr{]},
\end{equation}
the map $P$ can be written as
\begin{equation} \label{eq-T4}
P(H,L)=\sqrt{2} \int_{-\frac{\pi}{2}}^{\frac{\pi}{2}} 
\dfrac{1}{h'(r; H, L)}\Bigg{|}_{r=r(\vartheta;H,L)}\,\mathrm{d}\vartheta,
\end{equation}
for every $(H,L)\in \Lambda_E$.

From now on, we will use the notation 
\begin{equation*}
\Omega_{0}^{(i)} (H,L) = \Omega^{(i)}(1/r_{0}(H,L);H,L), \quad i\in \{0,1,2,3,4\}, \, (H,L)\in \Lambda_E.
\end{equation*}
and
\begin{equation}\label{def-omega0e}
\Omega_{0,E}^{(i)} = \Omega_{0}^{(i)}(-1/2,L_{E}), \quad i\in \{0,1,2,3,4\}.
\end{equation}

The following result, which is a straightforward modification of \cite[Lemma~A.3]{BoDaFe-21}, provides a formula for $\partial_L P (H,L)$.

\begin{lemma} \label{prop-mappetempo1}
Assume that
\begin{equation} \label{eq-ipoappendice}
\Omega^{(2)}_0(H,L)>0, \quad \text{for every $(H,L)\in \Lambda_E$.}
\end{equation}
Then, we have 
\begin{align}
\label{eq-dermappetempo}
&\partial_L P (H,L) =
\\
&= \dfrac{\partial_L \omega_0 (H,L)}{\sqrt{2}} \int_{-\frac{\pi}{2}}^{\frac{\pi}{2}} 
\dfrac{3 (h''(r; H,L))^{2} - h'(r; H,L) h'''(r; H,L)}{(h'(r; H,L))^{5}}  \Bigg{|}_{r=r(\vartheta;H,L)} \, \cos^{2} \vartheta  \, \mathrm{d}\vartheta 
\\
&\quad + \displaystyle \sqrt{2} \int_{-\frac{\pi}{2}}^{\frac{\pi}{2}} 
\dfrac{h''(r; H,L) \partial_{L}h(r; H,L) - h'(r; H,L)\partial_{L} h'(r; H,L)}{(h'(r; H,L))^{3}} \Bigg{|}_{r=r(\vartheta;H,L)}\, \mathrm{d}\vartheta, 
\end{align}
for every $(H,L)\in \Lambda_E$.
\end{lemma}

Taking advantage of the previous lemma, we can provide crucial asymptotic estimates on $P (-1/2,\cdot)$ as $L\to (L_E)^-$, which are on the lines of \cite[Proposition~A.1]{BoDaFe-21}.

\begin{lemma} \label{prop-mappetempo2}  
Assume \eqref{eq-ipoappendice}. The following relations hold true:
\begin{align*}
\lim_{L\to (L_E)^-} P\left(-\frac{1}{2},L\right) &= \dfrac{2\pi}{(\Omega^{(2)}_{0,E})^{\frac{1}{2}}},
\\
\lim_{L\to (L_E)^-} \partial_L P \left(-\frac{1}{2},L\right) &= \dfrac{\pi}{12}\, \dfrac{5\bigl{(}\Omega^{(3)}_{0,E}\bigr{)}^2-3\Omega^{(2)}_{0,E}\, \Omega^{(4)}_{0,E}}{\bigl{(}\Omega^{(2)}_{0,E}\bigr{)}^{\frac{7}{2}}}\, \partial_L \omega_0 (-1/2,L_E) -\pi \, \dfrac{\partial_L \Omega^{(2)}_{0,E}}{\bigl{(}\Omega^{(2)}_{0,E}\bigr{)}^{\frac{3}{2}}}.
\end{align*}
\end{lemma}

\begin{proof}
Let us first observe that, by the definition of $L_E$, from \eqref{eq-cambiovar} we deduce that
	\begin{equation*}
	\lim_{L\to (L_E)^-} r(\theta; -1/2,L)=0,
	\end{equation*}
for every $\theta \in [-\pi/2,\pi/2]$. Arguing as in the proof of \cite[Proposition~A.1]{BoDaFe-21}, it is possible to show that passing to the limit in the integrals in \eqref{eq-T4} and \eqref{eq-dermappetempo} is allowed. 

Now, we observe that a straightforward modification of the arguments in \cite[Lemma~A.1]{BoDaFe-21} proves that $h, h', h'', h'''\in \mathcal{C}(\Gamma_E)$, $\partial_H h, \partial_H h', \partial_L h, \partial_L h'\in \mathcal{C}(\Gamma_E)$ and that, for every $(H,L)\in \Lambda_E$, the relations
\begin{equation*}
\begin{aligned}
h'(1/r_{0}(H,L);H,L) &= \dfrac{(\Omega^{(2)}_{0}(H,L))^{\frac{1}{2}}}{\sqrt{2}},
\\
h''(1/r_{0}(H,L);H,L) &= \dfrac{\Omega^{(3)}_{0}(H,L)}{3\sqrt{2} \, ( \Omega^{(2)}_{0}(H,L))^{\frac{1}{2}}},
\\
h'''(1/r_{0}(H,L);H,L) &= \displaystyle \dfrac{3\Omega^{(2)}_{0}(H,L)\, \Omega^{(4)}_{0}(H,L)-(\Omega^{(3)}_{0}(H,L))^{2}}{ 12\sqrt{2} \, (\Omega^{(2)}_{0}(H,L))^{\frac{3}{2}}},
\end{aligned}
\end{equation*}
and
\begin{equation*}
\begin{aligned}
\partial_{H} h (1/r_{0}(H,L);H,L) &= -\dfrac{(\Omega^{(2)}_{0}(H,L))^{\frac{1}{2}}}{\sqrt{2}}\, \partial_H (1/r_0)(H,L),
\\
\partial_{H} h' (1/r_{0}(H,L);H,L) &= \dfrac{3 \partial_H \Omega^{(2)}_{0}(H,L) -2\Omega^{(3)}_{0}(H,L)\, \partial_H (1/r_0)(H,L)}{6\sqrt{2} \, (\Omega^{(2)}_{0}(L))^{\frac{1}{2}}},
\\
\partial_{L} h (1/r_{0}(H,L);H,L) &= -\dfrac{(\Omega^{(2)}_{0}(H,L))^{\frac{1}{2}}}{\sqrt{2}}\, \partial_L (1/r_0)(H,L),
\\
\partial_{L} h' (1/r_{0}(H,L);H,L) &= \dfrac{3 \partial_L \Omega^{(2)}_{0}(H,L) -2\Omega^{(3)}_{0}(H,L)\, \partial_L (1/r_0)(H,L)}{6\sqrt{2} \, (\Omega^{(2)}_{0}(L))^{\frac{1}{2}}}.
\end{aligned}
\end{equation*}
hold true. From the above formulas we deduce that 
\begin{align*}
&3(h''(1/r_{0}(H,L);H,L))^2-h'(1/r_{0}(H,L);H,L)\, h'''(1/r_{0}(H,L);H,L) =
\\
&=\dfrac{5(\Omega^{(3)}_{0}(H,L))^2-3\Omega^{(2)}_{0}(H,L)\, \Omega^{(4)}_{0}(H,L)}{24\Omega^{(2)}_{0}(H,L)},
\\
&h''(1/r_{0}(H,L);H,L)\, \partial_{H} h (1/r_{0}(H,L);H,L) 
\\
&\hspace{50pt}-h'(1/r_{0}(H,L);H,L)\, \partial_{H} h' (1/r_{0}(H,L);H,L)=
\\
&= -\dfrac{\partial_H \Omega^{(2)}_{0}(H,L)}{4},
\\
&h''(1/r_{0}(H,L);H,L)\, \partial_{L} h (1/r_{0}(H,L);H,L) 
\\
&\hspace{50pt}-h'(1/r_{0}(H,L);H,L)\, \partial_{L} h' (1/r_{0}(H,L);H,L)=
\\
&= -\dfrac{\partial_L \Omega^{(2)}_{0}(H,L)}{4}.
\end{align*}
Passing to the limit in \eqref{eq-T4}, with $H=-1/2$, we thus obtain
\begin{equation*}
\lim_{L\to (L_E)^-} P \left(-\dfrac{1}{2},L\right)= \dfrac{2\pi}{(\Omega^{(2)}_{0}(-1/2,L_{E}))^{\frac{1}{2}}},
\end{equation*}
which recalling the definition \eqref{def-omega0e} proves the first relation in the thesis. Analogously, passing to the limit in  \eqref{eq-dermappetempo}, again with $H=-1/2$, we obtain the second formula.
The proof is complete.
\end{proof}

We can now state and prove our final result on the map $\Theta$.

\begin{corollary} \label{cor-mappetempo}
For every $E\in (\sqrt{25/27},1)$, it holds that
\begin{equation} \label{eq-stimemappetempo2}
\displaystyle  \lim_{L\to (L_E)^-} \partial_L \Theta  \left(-\dfrac{1}{2},L\right) 
= \dfrac{-3\pi \left( 4 \sqrt{ \dfrac{L_E^2}{M^2} - 12} + \frac{5}{\sqrt{2}} \dfrac{L_E^2}{M^2}\zeta'\left(\sqrt{2} \dfrac{L_E}{M}\right)\right)}{M\sqrt{\frac{L_E}{M}} \, \left( \dfrac{L_E^2}{M^2} - 12\right)^{\frac{7}{4}}},
\end{equation}
where $\zeta$ is defined in \eqref{eq-defzetaausiliaria}. 
\end{corollary}

\begin{proof} 
Recalling \eqref{eq-thetaP}, we have
\begin{equation*}
\displaystyle \lim_{L\to (L_E)^-} \partial_L \Theta  \left(-\dfrac{1}{2},L\right)=\displaystyle \lim_{L\to (L_E)^-} P  \left(-\dfrac{1}{2},L\right)+ \displaystyle L_E \, \lim_{L\to (L_E)^-} \partial_L P \left(-\dfrac{1}{2},L\right).
\end{equation*}
From Lemma~\ref{prop-mappetempo2}, we then infer that 
\begin{equation} \label{eq-stima222}
\displaystyle \lim_{L\to (L_E)^-} \partial_L \Theta  \left(-\dfrac{1}{2},L\right)
= \dfrac{\pi}{12 (\Omega^{(2)}_{0,E})^{\frac{7}{2}}}\, \Theta_E,
\end{equation}
where
\begin{equation} \label{eq-stima333}
\Theta_E=24 (\Omega^{(2)}_{0,E})^3 +L_E \left((5(\Omega^{(3)}_{0,E})^2 
-3\Omega^{(2)}_{0,E}\, \Omega^{(4)}_{0,E})\, \partial_L \omega_0 (-1/2,L_E)
-12 (\Omega^{(2)}_{0,E})^2\, \partial_L \Omega^{(2)}_{0,E}\right).
\end{equation}
Notice that by now we have not checked that \eqref{eq-ipoappendice} holds true, but this will follows from \eqref{eq-stima444}.

Now, let us define
\begin{equation*}
v(x)=\dfrac{1}{2}\, x^2+\dfrac{1}{2}\, x\sqrt{x^2-12}, \quad x\in(4,+\infty),
\end{equation*}
and observe that \eqref{eq-defrpiu} implies that
\begin{equation*}
\dfrac{r_0(-1/2,L)}{M}=v(L/M).
\end{equation*}
On the other hand, from \eqref{eq-v0maxrif} we deduce that
\begin{equation*}
\omega_0(-1/2,L)=-\dfrac{1}{2}\, \zeta \left(\sqrt{2} L/M\right)
\end{equation*}
and then
\begin{equation*}
\partial_L \omega_0(-1/2,L)=-\dfrac{1}{\sqrt{2}}\, \dfrac{1}{M}\,  \zeta' \left(\sqrt{2} L/M\right).
\end{equation*}
Moreover, recalling \eqref{eq-defomegatrasl} and the fact that $\widetilde{W}(r; H,L)=W(1/r;H,L)$, simple computations yield
\begin{align} \label{eq-stima444}
&\Omega^{(2)}_{0}(-1/2,L) = M^2 (L/M) \sqrt{(L/M)^2-12},
\\
&\partial_L \Omega^{(2)}_{0}(-1/2,L) = 2M \frac{(L/M)^2-6}{\sqrt{(L/M)^2-12}},
\\
&\Omega^{(3)}_{0}(-1/2,L) = -6 M^3 (L/M)^2,
\\
&\Omega^{(4)}_{0}(-1/2,L) = 0.
\end{align}
Replacing in \eqref{eq-stima333}, we deduce
\begin{align}
\Theta_E 
&= 12 M^3 L_E^3 \biggl( 2 ((L_E/M)^2-12)^{3/2} - \frac{15}{\sqrt{2}} (L_E/M)^2 \zeta'(\sqrt{2} L_E/M) 
\\
&\quad - 2 ((L_E/M)^2-6)((L_E/M)^2-12)^{1/2} \biggr)
\label{eq-contothetaE}
\\ &=-36 M^3 L_E^3 \left( 4 ((L_E/M)^2-12)^{1/2} + \frac{5}{\sqrt{2}} (L_E/M)^2 \zeta'(\sqrt{2}L_E/M)\right).
\end{align}
From \eqref{eq-stima222}, \eqref{eq-stima444} and \eqref{eq-contothetaE} we obtain the thesis.
\end{proof}

We are now in a position to prove the final result of this section.

\begin{proposition} \label{teo-schwnonperturbato}
Let $E\in (\sqrt{25/27},1)$. Then, there exists a dense set $\mathcal{L}_E\subset (4M,L_E)$ such that for every $L\in \mathcal{L}_E$ there exists an equatorial timelike non-circular closed geodesics of angular momentum $L$ satisfying the non-degeneracy condition \eqref{eq-nondeg-Schw}.
\end{proposition}
	
\begin{proof} 
Let us fix $E\in (\sqrt{25/27},1)$ and consider the analytic function $\Theta (-1/2,\cdot)\colon$ $(4M, L_E)\to (0,+\infty)$.
From \eqref{eq-stimemappetempo2} and the fact that $\zeta'$ is positive, cf.~\eqref{eq-propzetaausiliaria}, it holds that
\begin{equation*}
\lim_{L\to (L_E)^-} \partial_L \Theta  \left(-\dfrac{1}{2},L\right) \neq 0,
\end{equation*}
and so the map $\Theta (-1/2,\cdot)\colon (4M, L_E)\to (0,+\infty)$ is not constant.
As a consequence, the set
\begin{equation*}
\mathcal{F}_E=\{L\in (4M, L_E)\colon \partial_L \Theta (-1/2,L)=0\}
\end{equation*}
is discrete and from this we easily infer that the set
\begin{equation*}
\mathcal{G}_E=\{L\in (4M, L_E)\colon \Theta (-1/2,L)\in 2\pi \mathbb{Q}^+\}.
\end{equation*}
is dense. Indeed, otherwise the complementary set
\begin{equation*}
(4M,L_E) \setminus \mathcal{G}_E=\{L\in (4M, L_E)\colon \Theta (-1/2,L)\in (0,+\infty) \setminus 2\pi \mathbb{Q}^+\}
\end{equation*}
would contain an open interval: clearly, the only possibility for this to happen is that $\Theta (-1/2,\cdot) \equiv \Theta^*$ on this interval, for some $\Theta^* \in  \mathopen{]}0,+\infty\mathclose{[} \setminus 2\pi \mathbb{Q}^+$, thus contradicting the fact that $\mathcal{F}_E$ is discrete.

So, let us define the set
\begin{equation*}
\mathcal{L}_E=\mathcal{G}_E\setminus \mathcal{F}_E.
\end{equation*}
Since $\mathcal{G}_E$ is dense and $\mathcal{F}_E$ is discrete, with a similar argument as before it is easily seen that $\mathcal{L}_E$ is dense as well. 
Moreover, by construction, for every $L\in \mathcal{L}_E$ we have both
\begin{equation*}
\Theta \left(-\dfrac{1}{2},L\right)\in 2\pi \mathbb{Q}^+ \quad \mbox{ and } \quad \partial_L \Theta  \left(-\dfrac{1}{2},L\right) \neq 0,
\end{equation*}
so that the corresponding timelike geodesic is closed, cf.~\eqref{eq-tori}, and satisfies the non-degeneracy condition \eqref{eq-nondeg-Schw}.
The proof is thus concluded.
\end{proof}

\subsection{The perturbed problem: bifurcation of equatorial timelike closed geodesics}\label{section-4.3}

Let us consider again the $4$-dimensional space-time endowed with the spherical coordinates \eqref{sph-coord}.
Given a \textit{stationary} and \textit{equatorial} metric $l$, i.e.~a metric of the form
\begin{equation*}
l_{11}(r,\phi)\, \mathrm{d}r^2 + 2l_{13}(r,\phi)\, \mathrm{d}r\, \mathrm{d}\phi+l_{33}(r,\phi) \, \mathrm{d}\phi^2,
\end{equation*}
where $l_{\mu\nu}\in \mathcal{C}^\infty ((0,+\infty)\times \mathbb{R})$, $\mu,\nu=1, 3$, is $2\pi$-periodic in the $\phi$-variable,
we consider the perturbed Schwartzschild metric $g^{\varepsilon} = g + \varepsilon l$, namely
\begin{equation}\label{eq-schwpert}
\begin{aligned}
\mathrm{d}s^2
=&-\alpha(r)\, \mathrm{d}\tau^2+\alpha(r)^{-1}\, \mathrm{d}r^2+r^2\, \mathrm{d}\theta^2+r^2\sin^2\theta \, \mathrm{d}\phi^2\\
&+\varepsilon \left(l_{11}(r,\phi)\, \mathrm{d}r^2+2l_{13}(r,\phi)\, \mathrm{d}r\, \mathrm{d}\phi+l_{33}(r,\phi) \, \mathrm{d}\phi^2\right),
\quad \varepsilon \in \mathbb{R}.
\end{aligned}
\end{equation}
The associated Lagrangian is
\begin{equation*}
L_{\varepsilon} (q,\dot{q})
= L_0 (q,\dot{q})+\dfrac{1}{2}\, \varepsilon \left(l_{11}(r,\phi)\, \dot{r}^2+2l_{13}(r,\phi)\, \dot{r}\, \dot{\phi}+l_{33}(r,\phi) \, \dot{\phi}^2\right),
\end{equation*}
where $L_{0}$ is as in \eqref{eq-lagrunpert} and again the derivative is taken with respect to the proper time $t$ of the particle. 
Accordingly, the canonical momenta are
\begin{equation*}
\begin{cases}
\, p_\tau =-\alpha(r)\, \dot{\tau},
\\
\, p_r =\alpha(r)^{-1}\, \dot{r}+\varepsilon (l_{11}(r,\phi)\, \dot{r}+l_{13}(r,\phi)\, \dot{\phi}),
\\
\, p_\theta =r^2\, \dot{\theta},
\\
\, p_\phi =r^2 \sin^2 \theta\, \dot{\phi}+\varepsilon (l_{13}(r,\phi)\, \dot{r}+l_{33}(r,\phi)\, \dot{\phi}).
\end{cases}
\end{equation*}
This system can be solved with respect to $(\dot{\tau},\dot{r},\dot{\theta},\dot{\phi})$ for small values of $\varepsilon$. More precisely, there exist $\hat{\varepsilon}>0$ and two functions $\eta_1, \eta_2 \in \mathcal{C}^\infty ((-\hat{\varepsilon},\hat{\varepsilon})\times (0,+\infty)\times \mathbb{R}\times \mathbb{R}^2\times [0,1])$, $2\pi$-periodic in $\phi$, such that for every $\varepsilon\in( -\hat{\varepsilon},\hat{\varepsilon})$ 
\begin{equation} \label{eq-momentiinv}
\begin{cases}
\, \dot{\tau} = -\alpha(r)^{-1}\, p_\tau,
\vspace{2pt}\\
\, \dot{r} =\alpha(r)\, p_r+\varepsilon \, \eta_1 (\varepsilon,r,\phi,p_r,p_\phi,\sin^2 \theta),
\vspace{2pt}\\
\, \dot{\theta} = \dfrac{1}{r^2}\, p_\theta,
\vspace{2pt}\\
\, \dot{\phi} =\dfrac{1}{r^2 \sin^2 \theta}\, p_{\phi}+\varepsilon \, \eta_2 (\varepsilon,r,\phi,p_r,p_\phi,\sin^2 \theta).
\end{cases}
\end{equation}
By Legendre transformation, the Hamiltonian $H_{\varepsilon}$ conjugate to $L_{\varepsilon}$ is
\begin{equation*}
H_{\varepsilon} (q,p) = \langle \dot{q}, p \rangle - L_{\varepsilon}(q,\dot{q})
\end{equation*}
where $\dot{q}=\dot{q}(\varepsilon,p, r,\phi,\sin^2 \theta)$ is given by \eqref{eq-momentiinv}, and a simple computation finally yields
\begin{equation*}
H_{\varepsilon} (q,p)= H_0(q,p) + \varepsilon R(\varepsilon,r,\phi,p_r,p_\phi,\sin^2 \theta),
\end{equation*}
with $H_0$ as in \eqref{eq-hamrunpert} and $R\in \mathcal{C}^\infty((-\hat \varepsilon,\hat \varepsilon)\times (0,+\infty)\times \mathbb{R}\times \mathbb{R}^2\times [0,1])$ and $2\pi$-periodic in $\phi$.

The Hamiltonian system corresponding to $H_{\varepsilon}$ is then 
\begin{equation} \label{eq-sisthamschw}
\begin{cases}
\, \dot{\tau} = \displaystyle -\alpha(r)^{-1}\, p_\tau,
\\
\, \dot{p_\tau} =  0,
\\
\, \dot{r} = \alpha(r)\, p_r+\varepsilon \, \partial_{p_r}R (\varepsilon,r,\theta, \phi,p_r,p_\phi,\sin^2 \theta),
\\
\, \dot{p_r} = -\dfrac{1}{2}\, \alpha'(r)\, \alpha(r)^{-2} \, p_\tau^2-\dfrac{1}{2}\, \alpha'(r)\, p_r^2 +\dfrac{1}{r^3}\, p_\theta^2+\dfrac{1}{r^3 \sin^2 \theta}\, p_\phi^2
\\
\, \qquad -\varepsilon\, \partial_{r}R(\varepsilon,r,\theta,\phi,p_r,p_\phi,\sin^2 \theta),
\\
\, \dot{\theta} = \dfrac{1}{r^2}\, p_\theta,
\\
\, \dot{p_\theta} = \dfrac{1}{r^2 \sin^3 \theta}\, \cos \theta \, p_\phi^2
- 2 \varepsilon^2  \partial_{\sin^{2} \theta} R(\varepsilon,r,\phi,p_r,p_\phi,\sin^2 \theta) \sin \theta \cos \theta,
\vspace{3pt}\\
\, \dot{\phi} = \dfrac{1}{r^2 \sin^2 \theta}\, p_{\phi}+\varepsilon \, \partial_{p_\phi} R (\varepsilon,r,\theta,\phi,p_r,p_\phi,\sin^2 \theta),
\\
\, \dot{p_\phi} =  - \varepsilon \partial_{\phi}R (\varepsilon,r,\theta, \phi,p_r,p_\phi,\sin^2 \theta),
\end{cases}
\end{equation}
From \eqref{eq-sisthamschw}, we immediately deduce that, as in the unperturbed case, $p_\tau$ is a first integral and $(\theta,p_\theta) \equiv (\pi/2,0)$ satisfies the pair of equations for $(\dot{\theta},\dot{p_\theta})$, for every $\varepsilon \in [-\hat \varepsilon,\hat \varepsilon]$. Still denoting by $E\in \mathbb{R}$ the value of the first integral $p_\tau$, we can consider again the $2$-degree of freedom system associated with the reduced Hamiltonian
\begin{equation*}
\mathcal{H}_{\varepsilon} (r,\phi,p_r,p_\phi)= \mathcal{H}_0 (r,\phi,p_r,p_\phi) + \varepsilon  R(\varepsilon,r,\phi,p_r,p_\phi,1),
\end{equation*}
where $\mathcal{H}_0$ is as in \eqref{eq-hamrunpertschwridotta0}.
Finally, the invariance of the mass of the particle in the case of timelike geodesics gives the isoenergetic condition
\begin{equation*}
\mathcal{H}_{\varepsilon} (r,\phi,p_r,p_\phi)=-\dfrac{1}{2}.
\end{equation*}
Recalling Proposition~\ref{teo-schwnonperturbato}, from the above discussion and Theorem~\ref{teo-main}, we immediately obtain the following result.

\begin{theorem} \label{teo-schwperturbato}
Let $E\in (\sqrt{25/27},1)$,  $L^*\in \mathcal{L}_E$ and $(n,k)$ coprime positive integers such that 
\begin{equation*}
\Theta \left(-\dfrac{1}{2},L^*\right)=2\pi \dfrac{k}{n}.
\end{equation*}
Then, there exist $\varepsilon^* > 0$ and $\delta^* > 0$ such that,
for every pair of positive coprime integers $(p,q)$ satisfying
\begin{equation*}
\left\vert \frac{p}{q} - \frac{k}{n}\right \vert < \delta^*,
\end{equation*}
and for every $\varepsilon$ such that $\vert\varepsilon \vert < \varepsilon^*$, there are two distinct equatorial timelike closed geodesics for the metric \eqref{eq-schwpert}. Moreover, the corresponding radial and angular components satisfy the qualitative properties given in Theorem~\ref{teo-main}.
\end{theorem}

\begin{remark} \label{rem-massless}
The dynamics of lightlike geodesics, which correspond to massless particles, is completely different. Indeed, in \cite{Xue-21} it is proved that the only bounded motions in the Schwarzschild spacetime are bound spherical orbits lying in a sphere corresponding to a maximum point of the effective potential (the so-called photon shell). These orbits are degenerate in the sense that the corresponding Hamiltonian in action-angles variables does not satisfy the non-degeneracy condition \eqref{eq-nondeg-Schw} (see \cite[Section~3.2]{Xue-21}). Hence, no general results can be a priori applied to ensure bifurcation of closed orbits for the perturbed metric.
\hfill$\lhd$
\end{remark}

\bibliographystyle{elsart-num-sort}
\bibliography{BoDaFe-biblio}

\end{document}